\documentclass{commat}

\newcommand{\Q}{\mathbb{Q}}
\newcommand{\tron}[1]{\left(#1\right)}
\newcommand{\QQ}{\mathbb{Q}} 
\newcommand{\RR}{\mathbb{R}} 
\newcommand{\ZZ}{\mathbb{Z}} 
\newcommand{\set}[1]{\left\{#1\right\}}
\newcommand{\Tr}{\text{Tr}}
\newcommand{\N}{\mathbb{N}}
\newcommand{\Gal}{\text{Gal}}
 
\usepackage{enumerate}

\title{%
	Well-Rounded ideal lattices of cyclic cubic and quartic fields}

\author{%
	Dat Tan Tran, Nam Hoai Le, Ha Thanh Nguyen Tran
}

\authorinfo[
Dat.T. Tran]{% Please, use "F. Name" format
	University of Science, Vietnam National University Ho Chi Minh City, Vietnam}{ttdat1323@gmail.com
}

\authorinfo[Nam H. Le]{% Please, use "S. Name" format
	University of Science, Vietnam National University Ho Chi Minh City, Vietnam}{   namhoai12to2@gmail.com
}

\authorinfo[
Ha T.N. Tran]{% Please, use "T. Name" format
	Concordia University of Edmonton, Canada}{%
	hatran1104@gmail.com
}

\VOLUME{31}
\NUMBER{2}
\YEAR{2023}
\firstpage{209}
\DOI{https://doi.org/10.46298/cm.11138}

\abstract{%
	In this paper, we find criteria for when cyclic cubic and cyclic quartic fields have well-rounded ideal lattices. We show that every cyclic cubic field has at least one well-rounded ideal. We also prove that there exist families of cyclic quartic fields that have well-rounded ideals and explicitly construct their minimal bases.  In addition,  for a given prime number $p$, if a cyclic quartic field has a unique prime ideal above $p$, then we provide the necessary and sufficient conditions for that ideal to be well-rounded.  Moreover, in cyclic quartic fields, we provide the prime decomposition of all odd prime numbers and construct an explicit integral basis for every prime ideal.
}



\begin{document}
	
	% Here is where the main text should be typed.
	
	% Please, consider the following suggestions while preparing your text:
	% (Following these suggestions will speed up the editorial process.)
	% * Avoid starting a new sentence with a mathematical formula;
	% * Try to separate adjacent formulas with words;
	% * Avoid inline formulas longer than half of a line. You can use math 
	%   displays (\[...\]) instead;
	% * Consider the use the enumerate and itemize environments for lists;
	% * Consider the use of \dots, \ldots, \dotsc, \cdot, etc, instead of "..." 
	%   or ".";
	% * Instead of numbering or citing an article by hand (using parenthesis or 
	%   brackets), consider the use of \cite, \ref and \eqref for citations and
	%   cross-references;
	% * Try to avoid inserting horizontal or vertical spacing, such as \hskip, 
	%   \vskip and \bigskip;
	% * Try to avoid inserting line or page brakes, such as \\, \newpage and
	%   \clearpage.

	\section{Introduction}
	
	A well-rounded (WR) ideal lattice or a WR ideal is an ideal of a number field for which the associated lattice is well-rounded. WR ideal lattices can be used to investigate various problems such as kissing numbers \cite{martinet2013perfect},  sphere packing problems \cite{JFC12, JDCS15}, and Minkowski’s conjecture  \cite{M05}. They also have a variety of applications to coding theory \cite{WR1,WR2}. 
	Previously, Fukshanksy et.~al.~proved results on WR ideals in real quadratic fields  \cite{FP12, FHLPSW13}, and Araujo and Costa obtained results on WR lattices (but not necessarily for WR ideals) of cyclic fields with degrees equal to odd primes \cite{DC19}. Generalizing this work, Damir and Mantilla-Soler \cite{DM20} construct a parametric family of WR sub-lattices of a tame lattice with a Lagrangian basis.  Another generalization of WR lattices are WR twists of ideal lattices which are investigated for real quadratic fields in \cite{DK19} and for imaginary fields in \cite{LTT22}. In \cite{S19}, it is shown that for any lattice $L$ there exists a diagonal real matrix $D$ with determinant equal to one and with positive entries
	such that $DL$ is WR. Further, \cite{DGAH18} provides an analysis of some WR lattices used in wiretap channels, and \cite{DKAGKH21} shows how to use WR lattices to optimize coset codes for Gaussian and fading wiretap channels.

	In this paper, we investigate WR ideals of cyclic cubic and cyclic quartic fields. In the cyclic cubic case, let $F$ be a cyclic cubic field with discriminant $\Delta_F$ and Galois group $\Gal(F) = \langle \sigma \rangle$. If a prime $p$ divides $\Delta_F$, it is ramified in $F$ and $p\mathcal{O}_F = P^3$ for a unique prime ideal $P$ and $\sigma^i(P)=P$ for $i\in\{0,1,2\}$. If $x$ is a shortest vector in $P$ and the set $\{\sigma^i(x):0\leq i\leq2\}$ is linearly independent, then $P$ is WR (see Definition \ref{def:WR}). This idea is not valid only for prime ideals: it also works for other ideals whose norms divide $\Delta_F$ (for example, ideals of the form $\prod_i P_i^{m_i}$ where $P_i$ are ramified prime ideals and $0 <  m_i \in \mathbb{Z}$). We can also do similarly for cyclic quartic fields with some modifications.
	
	\textbf{Our experiment:} 
	To implement the idea outlined above, we do the following: First, we find the defining polynomials of cyclic cubic and cyclic quartic fields. Using these polynomials together with Pari/GP \cite{PARI2}, we generate a list of all integral ideals of norms bounded by a certain number for each field. We then test which ideals in the list are WR by listing the shortest vectors of each ideal, using the function \texttt{qfminim}  in Pari/GP. We check if their conjugates form a set of rank $3$ in $\RR^3$ (for the cyclic cubic case) or rank $4$ in $\RR^4$ (for the cyclic quartic case). After identifying the WR ideals we examined their properties such as the geometry of their integral bases, the coordinates of shortest vectors with respect to a given integral basis, etc., and formulated conjectures. Finally, we proved these conjectures.
	
	\textbf{Our contributions:} Our main contribution is establishing the conditions for the existence of WR ideal lattices in cyclic number fields of degrees $3$ and $4$. For cyclic quartic fields, we consider both the real and complex cases. The results can be seen in Theorems \ref{thm:main1} -- \ref{thm:main6}. This is the first time such results are obtained for these classes of number fields. Further, we give families of cyclic cubic and cyclic quartic fields that admit WR ideals. We explicitly construct minimal integral bases of these ideals, which have applications in coding theory \cite{WR1, WR2}. Our other major contribution is that we provide the type decomposition of all odd primes in cyclic quartic fields (see Theorem \ref{theo:class_p}) and construct an explicit integral basis for every prime ideal (see Section \ref{sec:int_basis_ideal}).
	
	The results in Theorems \ref{thm:main1}, \ref{thm:theo_3}, \ref{theo:WR_condition_PIQJ}, \ref{thm:main_5}, and the one in Theorem \ref{thm:theo_2} where $3 \mid m$ are new and have not been studied before. The WR ideals presented in these theorems are generally not tame and are hence not mentioned in \cite{DM20}. In \cite{DC19}, WR ideals of quartic fields (found in Theorems \ref{theo:WR_condition_PIQJ} and \ref{thm:main_5}) and of cyclic cubic fields with $3 \mid m$ (found in Theorem \ref{thm:main1}.ii), Theorems \ref{thm:theo_2} and \ref{thm:theo_3}) are not investigated. 
	For the case of cyclic cubic fields where $3 \nmid m$, it has been showed that if $\frac{m}{4}\le q^2\le 4m$ then $Q$ is WR \cite[Theorem 4.1]{DC19}. For this last case, we used a different technique to prove that this condition is not only sufficient but also necessary (see Theorem \ref{thm:theo_2}). Moreover, the ideals in  Theorem \ref{thm:main1}.i) have larger norms,  $m^2$, which fall outside the range of $[m/4, 4m]$, and thus, they are distinct from those discussed in \cite[Theorem 4.1]{DC19}.
	
	We remark that in this paper, all the ideals are integral, and we only consider the well-roundedness of an ideal if it is primitive.

	The following theorem regarding cyclic cubic fields can be obtained from Propositions \ref{prop:Psquare_is_WR},  \ref{prop:cubic9dividem2} and \ref{prop:P0Psquare_WR}.
	
	\begin{theorem}\label{thm:main1}
		Every cyclic cubic field $F$ has orthogonal and WR ideal lattices. In particular,  denoting by $m$ the conductor of $F$, we have the following.
		\begin{enumerate}[i)]
			\item If $9\nmid m$, then the unique ideal of norm $m^2$ is orthogonal and WR.
			\item If $9\mid m$, then the unique ideal of norm $\frac{m^2}{27}$ is orthogonal and WR.
		\end{enumerate}	
	\end{theorem}
	\noindent Moreover, we obtain the following theorem by combining Propositions \ref{prop:cubic9ndividem}, \ref{prop:cubic9dividem1} and \ref{prop:P0PI_WR}.
	
	\begin{theorem}
		\label{thm:theo_2}
		Let $q$ be a square-free divisor of the conductor $m$ of a cyclic cubic field $F$. There is a unique ideal $Q$ of $\mathcal{O}_F$ such that $\textrm{N}(Q)=q.$ In this case, $Q$ is WR if and only if the following conditions are verified:
		\begin{itemize}
			\item $\frac{m}{4}\le q^2\le 4m$ when $3\nmid m$, or
			\item $3\mid q,\frac{m}{4}\le q^2\le 4m$ when $3\mid m$.
		\end{itemize} 
	\end{theorem}
	
	\noindent When the conductor of a cyclic cubic field is divisible by $9$, we have the following result (see Proposition \ref{prop:P0PIsqPj}).
	
	\begin{theorem}
		\label{thm:theo_3} Let $m= 9p_1p_2\cdots p_r(r\ge 2)$ and $q,q'$ be two coprime divisors of $p_1p_2\cdots p_r.$ The unique ideal of norm $3q^2q'$ is WR if and  only if $\frac{m}{36}\le qq'^2 \le \frac{4m}{9}$.
	\end{theorem}

	Combining Theorem \ref{theo:class_p},  Propositions \ref{prop:PIQ_J_d_even_notWR}, \ref{prop:bd_odd_notWR}, \ref{prop:dodd_b_ven_aplusb3_mod4}, \ref{lem:WR_aplusb1_1} and \ref{lem:WR_aplusb1_2}, one obtains the following theorem.
	\begin{theorem}\label{theo:WR_condition_PIQJ}
		Let $F$ be a cyclic quartic field defined by $a,b,c,d$ as in \eqref{eq:defpoly-quartic} and $p_I\mid d$, $q_J\mid a $ such that $d$ is a quadratic non-residue modulo $q$ for each prime divisor $q$ of $q_J$. Then there are unique ideals of norms $p_I$ and $q_J$, denoted by $P_I$ and $Q_J$ respectively.  Let 
		$$\mathcal{M}=\set{ 16q_J^2d,\: 8|a|d,\:  4q_I^2d+4|a|d,\:  16p_I^2q_J^2,\: 4p_I^2q_J^2+4|a|d,\: 4p_I^2q_J^2+4q_J^2d}.$$
		Then the ideal $P_IQ_J$  is WR if and only if 
		\begin{gather*}
			d\equiv 1\pmod 4, \: b\equiv 1 \pmod 2,\: a+b\equiv 1\pmod 4, \\
			\text{ and }\quad p_I^2q_J^2+q_J^2d+2|a|d\le\min\mathcal{M}.
		\end{gather*}
	\end{theorem}

	\begin{theorem}
		\label{thm:main_5} With the notation given in Theorem \ref{theo:WR_condition_PIQJ}, the following hold.
		\begin{enumerate}[i)]
			\item  The lattice $P_I$ is WR if and only if  $d\equiv 1\pmod 4,\: b\equiv 0 \pmod 2, \: a+b\equiv 1 \pmod 4$ and one of the following conditions is satisfied.
			\begin{itemize}
				\item  $|a|=1$ and $\frac{1}{5}d\le p_I^2\le 5d$,
				\item $|a|=3$ and $d\le p_I^2\le 9d$,
				\item $|a|=5$ and $\frac{7}{3}d\le p_I^2\le 5d$.   
			\end{itemize}	
			\item The lattice $Q_J$ is WR if and only if $d=5,\: b=2 , \: c = 1$ and $|a|\le q_J^2\le 5|a|.$ 
		\end{enumerate}
	\end{theorem}
	Note that the proof of Theorem \ref{thm:main_5} is presented after the proof of Proposition \ref{lem:WR_aplusb1_2}.

	For cyclic quartic fields $F$, considering any odd prime integer $p$, Theorem \ref{theo:class_p} provides a classification of classes of prime $p$ based on the ideal factorization of $p\mathcal{O}_F$. This can be done for $F$ because its defining polynomial (see in \eqref{eq:defpoly-quartic}) has the special form  $\tron{x^2-ad}^2 -a^2b^2d$. However, this has not been done for cyclic cubic fields since we do not know how their defining polynomials (see in \eqref{df-polynomial-cubic}) are factorized modulo an arbitrary prime.

	Let $p$ be any prime number. Based on the result of Theorem \ref{theo:class_p}, we can establish necessary and sufficient conditions on $p$ to have a unique prime ideal above $p$. Given this condition and by  Theorem \ref{thm:main_5}, we obtain conditions which are equivalent to the well-roundedness of these prime ideals as below.
	
	\begin{theorem}\label{thm:main6}
		Let $F$ be a cyclic quartic field defined by $a,b,c,d$ as in \eqref{eq:defpoly-quartic} and a prime $p$. There is a unique prime ideal of $\mathcal{O}_F$ above $p$ if and only one of the following conditions is satisfied.
		\begin{enumerate}[i)]
			\item The prime $p\mid d$.
			\item The prime $p\mid a$ and $d$ is a quadratic non-residue modulo $p$.
			\item The prime $p\nmid abcd$ and $d$ is a quadratic non-residue modulo  $p$.
		\end{enumerate}
		Moreover, let $P$ denote the unique prime ideal of $\mathcal{O}_F$ above $p$. Then $P$ is WR if and only if the conditions in Theorem \ref{thm:main_5} are satisfied.
	\end{theorem}
	
	Explicit minimal bases of these WR ideals can be seen in the above-mentioned propositions and Lemmas. Additionally, since $\Delta_F$ is given in \eqref{quarticdisc}, Theorem \ref{thm:main6} also tells us that if $\mathcal{O}_F$ has only one prime ideal $P$ above a given prime $p$, then $P$ being WR implies that $p\mid \Delta_F$.
	
	The structure of this paper is as follows. Section \ref{sec:bacground} serves to provide an initial review of WR ideal lattices and their properties, defining polynomials, integral bases, discriminants, and prime factorizations of ideals in cyclic cubic and cyclic quartic fields. We then investigate WR ideals of cyclic cubic fields in Section \ref{sec:cubic} and of cyclic quartic fields in Section \ref{sec:quartic}. Finally, in Section \ref{sec:conclusion}  we provide some conclusions and a conjecture related to WR ideals of these fields for future research.

	%%%%%%%%%%%%%%%%%%%%%%%%%%%%%%%%%%%%%%%%%%%%%%%%%%%%%%
	\section{Background}\label{sec:bacground}

	In this section, we will recall some fundamental knowledge about WR ideal latices, cyclic cubic, and cyclic quartic fields.

	%%%%%%%%%%%%%%%%%%%%%%
	\subsection{Well-rounded ideal lattices }
	
	Let $\mathcal{B}=\{v_1,v_2,\ldots,v_m\}$ be a linearly independent set of vectors in $\mathbb{R}^n$, $1\le m \le n$. The set $L=\left\{\sum_{i=1}^{m}a_iv_i|a_i\in \mathbb{Z}\right\}$ is called a \textit{lattice} in $\mathbb{R}^n$ of rank $m$ and the set $\mathcal{B}$ is said to be a \textit{basis} of $L$. In case $m=n$, we say that $L$ is a \textit{full rank lattice}.
	
	The value $|L| = \min_{0\ne u \in L}\|u\|^2$ is called the \textit{minimum norm} of the lattice $L\subset \mathbb{R}^n$, where $\|.\|$ denotes the usual Euclidean norm in $\mathbb{R}^n$, and the set of \textit{minimum vectors} of $L$ is defined as 
	\[S(L):=\{u\in L:\|u\|^2=|L|\}.\]
	
	\begin{definition} \label{def:WR}
		Let $L$ be a lattice in $\mathbb{R}^n$.
		\begin{enumerate}
			\item The lattice $L$ is \textbf{WR} if $S(L)$ generates $\mathbb{R}^n$, that is, if $S(L)$ contains $n$ linearly independent vectors.
			\item The lattice $L$ is said \textbf{strongly WR} if $S(L)$ consists of  a basis of $L$. In this case, we call this basis a \textit{minimal} basis of $L$. 
		\end{enumerate}
	\end{definition}
	For lattices in dimensions at most $3$ and most lattices in dimension $4$,  WRness and strong WRness are equivalent by \cite[Corollary 2.6.10]{martinet2013perfect}. 
	
	We denote by $B$ is an $n\times m$-matrix whose columns are the vectors of $\mathcal{B}$.
	\begin{definition}
		Let $L$ be a lattice of rank $n$ and its matrix basis $B$. The \textbf{determinant} of $L$, denoted by $\det(L)$, is defined $\det(L):=\sqrt{\det(B^TB)}$. In the special case that $L$ is a full rank lattice, $B$ is a square matrix, then we have $\det(L)=|\det(B)|$.
	\end{definition}
	The determinant of a lattice is well-defined since it is independent of our choice of basis $B$. Indeed,  $B_1$ and $B_2$ are two bases of $L$, if and only if $B_2=B_1U$ for some unimodular matrix $U$ with integer entries. Hence,
	\[\sqrt{\det(B_2^TB_2)}=\sqrt{\det(U^TB_1^TB_1U)}=\sqrt{\det(B_1^TB_1)}.\]
	
	We recall the following result.
	\begin{lemma} \label{sublatticeequal}
		Let $L$ and $L'$ be two full rank lattices in $\mathbb{R}^n$ ($n \ge 1$). Assume that $L' \subseteq L$ and $\det(L) = \det(L')$. Then $L'=L$.
	\end{lemma}
	\begin{proof}
		Let $B,B'$ be bases of $L,L'$, respectively. Suppose $B'=BA$, then 
		\[[L:L']=|\det(A)|=\frac{|\det(B')|}{|\det(B)|}=\frac{\det(L')}{\det(L)}=1.\]
		Hence, $L=L'.$ 
	\end{proof}
	
	Let $F$ be a number field of degree $n$ and signature $(r_1, r_2)$. Then $F$ has $r_1+r_2$ embeddings up to conjugation: $\sigma_1, \dots,  \sigma_{r_1+r_2}$ where the first $r_1$ of them are real, and the remaining $r_2$ are complex. We denote by $\Phi: F \hookrightarrow F \otimes \mathbb{R} \cong \mathbb{R}^{r_1} \times \mathbb{C}^{r_2}$ the map defined by $\Phi(f)= (\sigma_1(f), \cdots, \sigma_{r_1+r_2}(f))$. Here $\mathbb{R}^{r_1} \times \mathbb{C}^{r_2}$ is a Euclidean space with the scalar product: $\langle u, v\rangle = \sum_{i=1}^{r_1} u_i v_i + 2\sum_{i= r_1+1}^{r_2} \Re(u_i \overline{v_i}) $ where $\overline{v_i}$ is the complex conjugate of $v_i$.
	
	Let  $Q$ be a (fractional) ideal of $F$. Then it is known that $\Phi(Q)$ is a lattice in $\mathbb{R}^{r_1} \times \mathbb{C}^{r_2}$  by \cite{Bayer-Fluckiger99}. By identifying  $Q$ and $\Phi(Q)$, one has that  $Q$ is an ideal of $F$ and also a lattice in $\mathbb{R}^{r_1} \times \mathbb{C}^{r_2}$. Hence, we call ideals of $F$ ideal lattices, see  \cite{Bayer-Fluckiger99} and also \cite[Section 4]{Schoof08} for more details. An ideal lattice $Q$ is called WR if the lattice  $\Phi(Q)$ is WR. 
	
	\subsection{Cyclic cubic fields} \label{sec:cubicfields}

	Let $F$ be a cyclic cubic field with conductor $m$. By \cite[pp.6-10]{maki2006determination}, one has \begin{align}
		\label{conductor}m= \frac{a^2+3b^2}{4}
	\end{align} where $a$ and $b$ are integers satisfying one of
	the following conditions,
	\begin{itemize}
		\item $a \equiv 2 \pmod 3$,  \: $b \equiv 0 \pmod 3$ and $b>0$ for $3 \not| m$;
		\item $a \equiv 6 \pmod 9$, \: $b \equiv 3$ or $6 \pmod 9$ and $b>0$ for $3 | m$.
	\end{itemize}
	% \begin{align*} \label{eq:a_b_cubic}
		% a \equiv 2 \pmod 3, &\quad \quad b \equiv 0 \pmod 3&\text{ and } &b>0 \text{  for  } 3 \not| m, \text{ and}\\\nonumber
		% a \equiv 6 \pmod 9, &\quad \quad b \equiv 3 \text{  or  } 6 \pmod 9& \text{ and } &b>0 \text{  for  } 3 | m.
		% \end{align*}
	
	We recall that the conductor $m$ of $F$ has the form
	$$m = q_1 q_2 \cdots q_r,$$
	where $r \in \mathbb{Z}_{>0}$ and $q_1, \cdots , q_r$ are distinct integers from the set
	$$ \{9\} \cup  \{q: q\text{ is prime and } q\equiv 1\pmod 3\} = \{7, 9, 13, 19, 31, 37, \dots \}.$$
	The discriminant of $F$ is $\Delta_F=m^2$.
	See Hasse \cite{hasse1930arithmetische} for more details. From \cite{maki2006determination}, the following polynomial,  denoted by $df$, can be used to define $F$,
	\begin{align}\label{df-polynomial-cubic}
		df(x)=\left\{\begin{matrix}x^3 - x^2 + \frac{1-m}{3}x -\frac{m(a-3)+1}{27},& \text{if} \ 3 \not | \ m\\x^3 -\frac{m}{3}x -\frac{a m}{27},& \text{if}\  3 |m
		\end{matrix}\right..
	\end{align}
	Let $m = p_1 \cdots p_r$ or $m= 9 \cdot p_1 \cdots p_r$, where all the $p_i$ are distinct prime numbers congruent to $1$ modulo $3$. We arrange the $p_i$ such that $3=p_0 < p_1 < p_2 < \cdots < p_r$. 
	
	From now on, we denote by $\alpha$ a root of the defining polynomial $df(x)$ in \eqref{df-polynomial-cubic}.
	
	\begin{lemma}\label{lem:index-cubic}
		Let $id_3=[\mathcal{O}_F: \mathbb{Z}[\alpha]]$. Then $p_i$ does not divide the index $id_3$ for all $i\ge 0$.
	\end{lemma}
	\begin{proof}
		We suppose by contradiction that there exists $i\ge 0$ such that $p_i | id_3$. 
		By \eqref{df-polynomial-cubic}, we can calculate the discriminant of $df$ as
		$$\Delta_{df}= \frac{m^2(4m-a^2)}{27}.$$
		Since $F$ has discriminant $m^2$, one must have $id^2$ divides $\frac{4m-a^2}{27}$ or  equal to $\frac{4m-a^2}{27}$. Thus, $p_i^2$ divides $\frac{4m-a^2}{27}$. Moreover, $p_i|m$. It leads to $p_i^2|m$ which implies that $p_i=3$ since $3$ is the only prime of which square divides the conductor $m$ given in \eqref{conductor}. In other words, $3|m$ and hence $9$ divides $\frac{4m-a^2}{27} = \frac{b^2}{9} $ which is a contradiction since $b \equiv 3 \text{  or  } 6 \pmod 9$ in \eqref{conductor}. Thus, $p_i \not| id_3$ for all $i$.
	\end{proof}
	
	We prove the following.
	\begin{lemma}\label{lemindependentcubic}
		Let $g\in \mathcal{O}_F\backslash \mathbb{Z}$. Then $\text{Tr}(g) \ne 0$ if and only if $\set{g, \sigma(g), \sigma^2(g)}$ is $\mathbb{R}$-linearly independent.
	\end{lemma}
	\begin{proof}
		It is implied from the following equality\begin{align*}
			\begin{vmatrix}
				g&\sigma(g)&\sigma^2(g)\\\sigma(g)&\sigma^2(g)&g\\ \sigma^2(g)&g&\sigma(g)
			\end{vmatrix}= \begin{matrix}
				& \\   -\frac{1}{2}(g+\sigma(g)+\sigma^2(g))\big((g-\sigma(g))^2 &  \\ & \hspace{-1.5cm} +(\sigma(g)-\sigma^2(g))^2+(\sigma^2(g)-g)^2\big).
			\end{matrix} 
		\end{align*}
	\end{proof}
	
	\subsection{Cyclic quartic fields}\label{sec:quarticfields}
	We first recall the facts about cyclic quartic fields and their properties. See \cite{hudson1990integers} for more details. Let
	$F=\mathbb{Q}(\beta)$
	where
	$a, b, c, d$ are integers such that 
	$a$ is squarefree and odd,
	$d= b^2+c^2$ is squarefree,
	$b>0, c>0$,
	$\gcd(a, d)=1$ and $\beta=\sqrt{a(d - b \sqrt{d})}$. If $a>0$ then $F$ is a totally real cyclic quartic field. If $a<0$ then $F$ is a totally imaginary cyclic quartic field.
	
	A defining polynomial of $F$, which is also the minimum polynomial of $\beta$, is 
	\begin{equation}\label{eq:defpoly-quartic}
		df(x)= x^4 -2ad x^2 + a^2c^2d.
	\end{equation}
	It is easy to verify that the discriminant of $df(x)$ is $\Delta_{df}=256a^6b^4c^2d^3$ and by \cite{hudson1990integers}, the discriminant of $F$ is 
	\begin{align}\label{quarticdisc}
		\Delta_F= %\left\{
		\begin{cases}
			2^8a^2d^3 &\text{ if } d\equiv 0\pmod 2,  \\
			2^6a^2d^3 & \text{ if } d\equiv 1\pmod2,\:b\equiv 1\pmod 2, \\ 
			2^4a^2d^3 &\text{ if } d\equiv 1\pmod2,\:b\equiv 0\pmod 2,\: a+b\equiv 3\pmod 4,\\ 
			a^2d^3 &\text{ if } d\equiv 1\pmod2,\:b\equiv 0\pmod 2, \:a+b\equiv 1\pmod 4.\\ 
		\end{cases}%\right.
	\end{align}
	Let $id_4$ be the index of $\mathbb{Z}[\beta]$ in $\mathcal{O}_F$. Then, by \eqref{quarticdisc}, $id_4^2$ divides the following quantity
	%This implies that the index $id_4$ of $\ZZ[\beta]$ in $\mathcal{O}_F$ is given as follows
	\begin{align}\label{eq:index_quartic_cyclic}
		\frac{\Delta_{df}}{\Delta_F} =%\left\{
		\begin{cases}
			a^2b^2c &\text{ if } d\equiv 0\pmod 2,  \\
			2a^2b^2c & \text{ if } d\equiv 1\pmod2,\:b\equiv 1\pmod 2, \\ 
			2^2a^2b^2c &\text{ if } d\equiv 1\pmod2,\:b\equiv 0\pmod 2, \:a+b\equiv 3\pmod 4,\\ 
			2^4a^2b^2c &\text{ if } d\equiv 1\pmod2,\:b\equiv 0\pmod 2, \:a+b\equiv 1\pmod 4.\\ 
		\end{cases}%\right.
	\end{align}
	For $K = \mathbb{Q}(\sqrt{d})$, we always have the tower of field extensions \begin{align}
		\label{eq:exten_tower} \mathbb{Q} \leq K \leq F.
	\end{align}
	The field $F$ has four embeddings: $1, \sigma, \sigma^2, \sigma^3$ where 
	\begin{equation}\label{embbeddings}
		\sigma: \beta \longmapsto \sigma(\beta), \qquad \sigma(\beta) \longmapsto -\beta, \qquad \sqrt{d} \longmapsto -\sqrt{d}.
	\end{equation}
	
	In case $a<0$, the field $F$ is totally complex and  the four roots of $df(x)$ are the following: $\beta, -\beta, \sigma(\beta)=\sqrt{a(d+b \sqrt{d})},  -\sigma(\beta)$, which are all in $\mathbb{R} i$.  Here one has $\overline{1}=  \sigma^2$ and $\overline{\sigma}=  \sigma^3$.  Thus $F$ has two embeddings $1$ and $ \sigma$ up to conjugation. For $\delta\in F$, we embed it to $(\delta, \sigma(\delta) ) \in \mathbb{C}^2$ which is then can be viewed as $(\Re(\delta), \Im(\delta), \Re(\sigma(\delta) ), \Im(\sigma(\delta) )) \in \mathbb{R}^4$. The four roots of $df(x)$ are totally imaginary hence, they have the form $(0, z_1, 0, z_2)$ for some $z_1, z_2 \in \mathbb{R}$ when embedded in $\mathbb{R}^4$.
	
	In case  $a>0$, the field $F$ is  totally real   and the 4 roots of $df(x)$ are the following: $\beta, -\beta, \sigma(\beta)=\sqrt{a(d+b \sqrt{d})},  -\sigma(\beta)$, which are all in $\mathbb{R}$. When we embed an element $\delta\in F$ in $\mathbb{R}^4$, we obtain the vector $ (\delta,\sigma(\delta),\sigma^2(\delta),\sigma^3(\delta))$. 
	
	Although the embeddings of imaginary and the totally real fields are different, we can  still verify that if $\delta=s_1+s_2\sqrt{d}+s_3\beta+s_4\sigma(\beta) \in F$ where $s_i \in \mathbb{Q}$ for all $i \in \{1, 2, 3, 4\}$, then 
	\begin{align}
		\label{length} \|\delta\|^2=4\left(s_1^2+s_2^2d+|a|ds_3^2+|a|ds_4^2\right).
	\end{align} In particular,
	\[\|\beta\|^2=4|a|d.\]
	
	\begin{remark}\label{rem:integralbasis}
		The following integral basis $\mathcal{B}=\left\{\gamma_1',\gamma_2',\gamma_3',\gamma_4'\right\}$ in this order of $F$ is provided in \cite{hudson1990integers}  which we will use in the later sections.
		\begin{enumerate}[i)]
			\item\label{rem:integralbasis1} $\set{1,\sqrt{d},\sigma(\beta),\beta}$, if $d\equiv 0 \pmod 2$;
			\item $\left\{1,\frac{1}{2}(1+\sqrt{d}),\sigma(\beta),\beta\right\}$, if $d\equiv b\equiv 1 \pmod 2$;
			\item $\set{1,\frac{1}{2}(1+\sqrt{d}),\frac{1}{2}(\sigma(\beta)+\beta),\frac{1}{2}(\sigma(\beta) -\beta)}$, if \[d\equiv 1\pmod 2,\:b \equiv 0 \pmod 2,\:a+b\equiv 3\pmod 4;\]
			\item\label{rem:integralbasis4} $\set{1,\frac{1}{2}(1+\sqrt{d}),\frac{1}{4}(1+\sqrt{d}+\sigma(\beta)-\beta),\frac{1}{4}(1-\sqrt{d}+\sigma(\beta) +\beta)}$, if \[d\equiv 1\pmod 2,\:b\equiv 0 \pmod 2,\:a+b\equiv 1 \pmod 4,\:a\equiv -c\pmod 4;\]
			\item $\set{1,\frac{1}{2}(1+\sqrt{d}),\frac{1}{4}(1+\sqrt{d}+\sigma(\beta)+\beta),\frac{1}{4}(1-\sqrt{d}+\sigma(\beta) -\beta)}$, if \[d\equiv 1\pmod 2,\:b\equiv 0 \pmod 2,\:a+b\equiv 1 \pmod 4,\:a\equiv c\pmod 4;\]		
		\end{enumerate}
		where $\beta=\sqrt{a(d-b\sqrt{d})}$ and $ \sigma(\beta)=\sqrt{a(d+b\sqrt{d})}$. 
		
	\end{remark}
	
	\begin{lemma}\label{lemindependentquartic}
		Let $\delta\in \mathcal{O}_F\backslash \mathbb{Z}$. Then $\text{Tr}(\delta) \ne 0$ if and only if the set $\set{\delta, \sigma(\delta), \sigma^2(\delta ), \sigma^3(\delta)}$ is $\mathbb{R}$-linearly independent.
	\end{lemma}
	
	\begin{lemma}
		\label{lem:norm_some_ele}One has the following results. 
		\begin{align}\label{eq:norm}
			\N\tron{\sqrt{d}} = d^2, \quad
			\N\tron{\beta }=a^2c^2d, \quad \N\tron{\frac{\beta+\sigma(\beta)}{2}} = \frac{a^2b^2d}{4},
		\end{align}
		\begin{align}\label{eq:mul_sqrt_d}
			\beta\cdot\sqrt{d} = c \sigma(\beta)-b\beta, \quad
			\sigma(\beta)\cdot\sqrt{d} = c\beta +b\sigma(\beta).
		\end{align}
	\end{lemma}
	\begin{proof}
		It is easy to verify all equalities in \eqref{eq:norm}. Hence, we only claim two equalities in \eqref{eq:mul_sqrt_d}. It is sufficient to show that $\beta\sqrt{d} = c\sigma(\beta)-b\beta$. Indeed, one has \begin{align*}
			\beta\sqrt{d}& = \frac{acd}{\sigma(\beta)} = c \frac{\beta^2+\sigma(\beta)^2}{2\sigma(\beta)}= c \frac{\tron{\beta+\sigma(\beta)}^2-2\sigma(\beta)\beta}{2\sigma(\beta)}= c\frac{\tron{\beta+\sigma(\beta)}^2}{2\sigma(\beta)}-c\beta.
		\end{align*}   Moreover, 
		\[c\frac{\tron{\beta+\sigma(\beta)}^2 }{2\sigma(\beta)}=c\frac{ad+ac\sqrt{d}}{\sqrt{ad+ab\sqrt{d}}}=c\sigma(\beta)+ (c-b)\beta. \]
		Therefore $\beta\sqrt{d}= c\sigma(\beta)-b\beta$. 
	\end{proof}
	\newcommand{\gene}[1]{\langle #1\rangle }
	\begin{lemma}\label{lem:quart_int_bas}
		Let $a,b,c,d,F,\beta$ in \eqref{df-polynomial-cubic} and $p$ be a prime number. Then:
		\begin{enumerate}[i)]
			\item\label{idealPi0quartic} If $p\mid d$ then $p\mathcal{O}_F= P^4 $ where $P =\gene{p,\beta}$ is the unique prime ideal of $\mathcal{O}_F$ above $p$.
			\item \label{lem:quart_int_bas_ii}  Assume that $p$ is odd, $p\nmid abcd$ and $d$ is not a quadratic residue modulo $ p$. Then $p\mathcal{O}_F$ is inert in $F$.
			\item\label{lem:quart_int_bas_iii} Assume that $p$ is odd, $p\nmid abcd$ and $d$ is a quadratic residue modulo $ p$ %and denote by $z$ the number such that 
			where $z$ is such that $d \equiv z^2 \pmod p$.
			
			If $ad+abz,ad-abz$ are quadratic residues modulo $p$ with
			\[ad+abz \equiv t_1^2, \quad ad-abz \equiv t_2^2\pmod p,\]
			then $p\mathcal{O}_F$ totally splits in $F$, i.e, $p\mathcal{O}_F=P_1P_2P_3P_4$ where 
			\[P_1 = \gene{p,\beta +t_1 },\: P_2 =\gene{p,\beta-t_1 },\: P_3 = \gene{p, \beta +t_2},\: P_4 = \gene{p, \beta -t_2}\]
			are all prime ideals of $\mathcal{O}_F$ above $p$. Otherwise, $p\mathcal{O}_F=P_1P_2$ where 
			\[P_1 =  \gene{p, abz+ab\sqrt{d}},\: P_2 = \gene{p, abz-ab\sqrt{d}}\]
			are all prime ideals above $p$.   
		\end{enumerate}
	\end{lemma}
	\begin{proof}
		In all the above cases, the prime $p$ is not a divisor of index $id_4$ (see \eqref{eq:index_quartic_cyclic}). By using the result on the decomposition of primes \cite[Theorem 4.8.13]{cohen1993course}, the prime composition of $p\mathcal{O}_F$ can be obtained by factorizing $df(x)$ over $\ZZ_p$. Note that $df(x)=  \tron{x^2-ad}^2-a^2b^2d$. \begin{enumerate}[i)]
			\item If $p\mid d$, then $df(x)\equiv  x^4 \pmod p$ and thus $p\mathcal{O}_F = P^4$ where $P= \gene{p,\beta}$ and $P$ is a unique prime ideal above $p$.
			\item  To prove \ref{lem:quart_int_bas_ii}, it is sufficient to prove $df(x)$ is irreducible in $\ZZ_p[x]$. By contradiction, suppose that the polynomial $df(x)$ is reducible over the field $F_p$. Since $d$ is not a quadratic residue modulo $p$, $df(x)$ has no root in $\ZZ_p$. We now claim that $df(x)$ cannot be decomposed into the product of two quadratic polynomials. 
			Indeed, if 
			\[df(x)\equiv  \tron{x^2+Ax+B}\tron{x^2+Cx+D}\pmod p,\]
			then 
			\[A+C \equiv0,\: B+D+AC \equiv 2ad,\: AC+BD\equiv 0,\: BD \equiv  a^2c^2d\pmod p.\]
			This implies that 
			\[A \equiv -C,\:C\tron{B-D} \equiv 0,\:BD\equiv a^2c^2d\pmod p.\] The integer $C$ must be nonzero because otherwise $BD=0=a^2c^2d$ and thus $p\mid acd$, contradicting the assumption that $p\nmid abcd$. 
			
			From $A\equiv -C\pmod p,\: C\tron{B-D} \equiv 0\pmod p$ and $C$ being nonzero, one obtains $B\equiv D\pmod p$ and thus $D^2  \equiv BD \equiv   a^2c^2d\pmod p$, which also contradicts the fact that $d$ is a quadratic non-residue modulo $ p$. This means $df(x)$ is irreducible over $\ZZ_p$ and hence $p\mathcal{O}_F$ is prime.
			\item One has \begin{align*}
				df(x) \equiv \tron{x^2-ad}^2-a^2b^2z^2 \equiv  \tron{x^2-\tron{ad-abz}}\tron{x^2-\tron{ad+abz}}\pmod p.
			\end{align*}
			If $x^2-\tron{ad-abz}$ and $x^2-\tron{ad+abz}$ are irreducible  over $
			F_p$, then  $p\mathcal{O}_F = P_1P_2$ where $P_1 = \gene{p, abz+ab\sqrt{d}},P_2 =\gene{p,abz-ab\sqrt{d}}$.
			Otherwise, 
			\[df(x) \equiv \tron{x-t_1}\tron{x+t_1}\tron{x-t_2}\tron{x+t_2}\pmod p.\]
			Thus, $p\mathcal{O}_F = P_1P_2P_3P_4$ where 
			\[P_1 = \gene{p,\beta +t_1 },\: P_2 =\gene{p,\beta-t_1 },\: P_3 = \gene{p, \beta +t_2},\: P_4 = \gene{p, \beta -t_2}\]
			are all prime ideals of $\mathcal{O}_F$ above $p$. 
		\end{enumerate}   
	\end{proof}
	In the Lemmas \ref{lem:quartic_int_basis_divisor_index} and \ref{lem:quartic_int_basis_divisor_b}, we will consider prime divisors of the index of the field. In these cases, we cannot apply the result on the decomposition of primes \cite[Theorem 4.8.13]{cohen1993course}, instead, we can apply \cite[Proposition 6.2.1]{cohen1993course}.
	
	\begin{lemma}	\label{lem:quartic_int_basis_divisor_index}
		Let $a,b,c,d, F, K,\beta$ be as in \eqref{df-polynomial-cubic} and $p$ be a prime number. Then:
		\begin{enumerate}[i)]
			\item Assume $p\mid a$. If $d$ is a quadratic non-residue modulo $p$, then there is a unique prime ideal $P$ above $p$ and $p\mathcal{O}_F= P^2.$ If $d$ is a quadratic residue modulo $p$ then there are exactly two prime ideals $P_1,P_2$ above $p$ and $p\mathcal{O}_F= P_1^2P_2^2.$
			\item Assume $p\mid c$ and $p\nmid a$. If $2a$ is a quadratic non-residue modulo $ p$, then there are exactly two prime ideals $P_1,P_2$ above $p$ and  $p\mathcal{O}_F=  P_1P_2$. In this case, $P_1 = \gene{p,  ad - ab\sqrt{d}}$ and $ P_2 = \gene{p, ad +ab\sqrt{d}}$. Otherwise, let $2a \equiv l^2\pmod p$. Then $p\mathcal{O}_F=P_1P_2P_3P_4$ where $P_1,P_2,P_3,P_4$ are all prime ideals of $\mathcal{O}_F$ above $p$, with
			\[P_1 = \gene{p, \beta - lb},\: P_2 = \gene{p,\beta+lb}, \: P_3P_4=  \gene{p, \beta^2},\] 
			and each ideal $P_3$ and $P_4$ is coprime with the ideals $P_1$ and $P_2$.  
		\end{enumerate}
	\end{lemma}
	\begin{proof} 
		\begin{enumerate}[i)]
			\item 	 We have $p|a$, hence $p|\Delta_F$ by \eqref{quarticdisc}. Thus, $p$ is ramified in $F$, i.e., one has that the prime decomposition of  $p\mathcal{O}_F$ is of the form 
			$P^4,\:P_1^2Q P_2^2$ or $P^2$ 
			where $P,P_1,P_2$ are prime ideals above $p$ of $\mathcal{O}_F$ since $F$ is Galois. On the other hand, we have $\gcd(a,d)=1$, so $p\nmid d$ and $\genfrac{(}{)}{}{}{d}{p} \ne 0$ and therefore $p$ is unramified in $K=\mathbb{Q}(\sqrt{d})$. As a result, $p\mathcal{O}_F$ is not of the form $P^4$ but instead $p\mathcal{O}_F=P_1^2P_2^2$ or $p\mathcal{O}_F=P^2$.
			Now, if $d$ is a quadratic residue modulo $ p$, it implies that $p$ splits in $K$. It follows that $p\mathcal{O}_F=P_1^2P_2^2$.
			In the other cases, $d$ is not a quadratic residue modulo $ p$, which implies $p$ is inert in $K$ and hence $p\mathcal{O}_F=P^2$.
			
			\item If $p\mid c$ then $d \equiv b^2 \pmod p, b\not\equiv 0 \pmod p$ and thus $df(x)\equiv x^2\tron{x^2-2ad} \pmod p$. If $2a$ is a quadratic non-residue modulo $ p$, then $x^2-2ad$ is irreducible modulo $ p$. By \cite[Proposition 6.2.1]{cohen1993course}, we have $p\mathcal{O}_F = P_1P_2$ where $P_1 = \gene{p,\beta^2},\: P_2 =  \gene{p, \beta^2-2ad}$ and $P_1,P_2$ are co-prime. Since $x^2-2ad$ is irreducible, $P_2$ is prime, and thus $P_1$ is also prime as $F$ is Galois. Hence, there are only two prime ideals of $\mathcal{O}_F$ above $p$, namely $P_1$ and $P_2$. Similarly, considering the remaining case and by  \cite[Proposition 6.2.1]{cohen1993course}, one has that $df(x)= \tron{x-lb}\tron{x+lb}x^2$ and $p\mathcal{O}_F = P_1P_2A$ where $P_1 = \gene{p, \beta - lb},P_2 = \gene{p,\beta+lb}$, $A=  \gene{p, \beta^2}$. Moreover, \cite[Proposition 6.2.1]{cohen1993course} also yields that $P_1,P_2$ are prime and due to the Galois property of $F$, $A = P_3P_4$. 
		\end{enumerate}
	\end{proof}
	\begin{lemma}
		\label{lem:quartic_int_basis_divisor_b}
		Let $a,b,c,d,F,K,\beta$ be as in \eqref{df-polynomial-cubic} and $p$ be an odd prime divisor of $b$ such that $p\nmid a$. Then:
		\begin{enumerate}[i)]
			\item If $a$ is a quadratic non-residue modulo $ p$, then there are at most two prime ideals above $p$ in $\mathcal{O}_F$ and $p\mathcal{O}_F$ is equal to the product of these prime ideals. 
			\item If $a$ is a quadratic residue modulo $ p$, then there are at least two prime ideals above $p$ in $\mathcal{O}_F$ and $p\mathcal{O}_F$ is equal to the product of these prime ideals.
		\end{enumerate}
	\end{lemma}
	\begin{proof}
		One has $df(x)\equiv \tron{x^2-ad}^2\pmod p$. We consider the first case in which $x^2-ad$ is irreducible. According to \cite[Proposition 6.2.1]{cohen1993course}, if $P$ is a prime ideal such that $P\mid p\mathcal{O}_F$, then $\N\tron{P} = p^m$ where $m\ge 2$. This implies that $p\mathcal{O}_F$ is a product of at most two prime ideals. In the remaining case, $df(x)$ is the square of the product of two linear polynomials. By using \cite[Proposition 6.2.1]{cohen1993course}, $p\mathcal{O}_F$ is a product of two nontrivial coprime ideals. Hence, there are at least two prime ideals in the prime decomposition of $p\mathcal{O}_F$.     
	\end{proof}
	The following lemma tells us the factorization of $2\mathcal{O}_F$ when $\Delta_F$ is even. In the case where $\Delta_F$ is odd, the factorization of $2\mathcal{O}_F$ will have one of the three forms: $P$, $P_1P_2$, or $P_1P_2P_3P_4$.
	
	\begin{lemma}
		\label{lem:prime_ideal_2} Assume that $2 \mid \Delta_F$. Then: \begin{enumerate}[i)]
			\item If $d$ is even, then there exists a unique prime ideal $P_0$ above $2$ and $\N\tron{P_0}=2$.
			\item If $d \equiv 5 \pmod 8$, then there exists a unique prime ideal $P_0$ above $2$ and $\N\tron{P_0}=4$.
			\item If $d \equiv 1 \pmod 8$, then $2\mathcal{O}_F= P_1^2P_2^2$, where $P_1,P_2$ are two distinct prime ideals and $\N\tron{P_1}= \N\tron{P_2}=2$. 
		\end{enumerate}
		
	\end{lemma}
	\begin{proof}
		If $d$ is even, then by Lemma \ref{lem:quart_int_bas}.i), one has that $P_0 = \langle 2,\beta\rangle $ is a unique prime ideal above 2. If $d\equiv1,5\pmod 8$, then $2$ ramifies in $F$ since $2 \mid  \Delta _F$. Thus, the factorization of $2\mathcal{O}_F$ has one of the forms $ R^4, R_1^2R_2^2, R^2$ for some prime ideals $R,R_1,R_2$ above $2$ (since $F$ is Galois). Let $K = \QQ\tron{\sqrt{d}}$. Then 
		\[df_K(x) = x^2 -x-\frac{d-1}{4}\]
		is a defining polynomial of $K$ and $2$ does not ramify in $K$. Hence $2\mathcal{O}_F\ne R^4$. In the case where $d\equiv 5 \pmod 8$,  $df_K(x)$ is irreducible modulo $ 2$ and thus $2$ is inert in $K$. Hence $2\mathcal{O}_F =R^2$. If $d\equiv 1 \pmod 8$ then $df_K(x)$ is reducible modulo $ 2$ and thus $2$ splits in $K$.  Hence $2\mathcal{O}_F = P_1^2P_2^2.$ 
	\end{proof}
	
	\section{Well-rounded ideal lattices of cyclic cubic fields} \label{sec:cubic}
	
	Let $F$ be a cyclic cubic field with conductor $m$. In this section, we will find WR ideals of $F$ and compute minimal bases of these ideals. 
	
	We denote by $P_i$ the unique prime ideal above the prime $p_i\mid m$ for each $i\ge 0$ and $\alpha$ a root of the defining polynomial $df(x)$ as in \eqref{df-polynomial-cubic}. We will fix these notations for the whole section. 
	
	\subsection{The case \texorpdfstring{$9\nmid m$}{9∤m}}
	Let $m= p_1 \cdots p_r$ with $7 \le p_1 < p_2 \cdots < p_r$, and $p_i \equiv 1 \pmod 3$ for all $i$ and $r \ge 1$. In this section, we will show that:
	%enumerate
	
	1) the ideal $(P_1\cdots P_r)^2$ is orthogonal and WR -- this result has not been proven before; and
	
	2) if $I\subset \{1,\ldots,r\}$, then $\prod_{i\in I}P_i$ is WR if and only if $\frac{m}{4}\le \tron{\prod_{i\in I}p_i}^2\le 4m$.
	
	\begin{lemma}\label{integralbasis-3notdiv9}
		The sets $\{\alpha, \sigma(\alpha), \sigma^2(\alpha)\}$ and $\{1, \alpha, \sigma(\alpha)\}$  are two integral bases of $\mathcal{O}_F$.
	\end{lemma}
	From now on, we will use one of the integral bases as mentioned in Lemma \ref{integralbasis-3notdiv9} depending on which one is convenient for our calculation.

	By \cite[page 166]{narkiewicz1974elementary}, also \cite[page 2]{de2017integral} and by \cite{maki2006determination}, we obtain Lemma \ref{lencoeff}.
	
	\begin{lemma}\label{lencoeff}
		Let $z = z_1 \alpha + z_2  \sigma(\alpha)  + z_3 \sigma^2(\alpha) \in \mathcal{O}_F$ where $z_i \in \mathbb{Z}, 1 \le i \le 3$. Then 
		$$\|z\|^2 = \text{Tr}(z^2)= m (z_1^2 + z_2 ^2 + z_3^2) + \frac{(1-m)(z_1+z_2+z_3)^2}{3}.$$
		Moreover, one can rewrite this expression as 
		$$\|z\|^2= \frac{m}{3}\left( (z_1- z_2)^2 + (z_2-z_3)^2 + (z_3- z_1)^2 \right) + \frac{1}{3}(z_1+z_2+z_3)^2.$$
	\end{lemma}

	Since $\alpha$ is a root of the defining polynomial $df(x)$ (see \eqref{df-polynomial-cubic}), using \cite[Proposition 2.2]{TranPeng1}, one can show that $\text{Tr}(\alpha)=1$ and  $\alpha$ is a shortest vector in $\mathcal{O}_F\backslash \mathbb{Z}$ with $\|\alpha\|^2 = \frac{2m+1}{3}$. 
	
	For $\ell \in \mathbb{Z}$ and $\ell>0$, as in \cite{DC19}, we define
	$$M_{\ell} = \{z = z_1 \alpha + z_2  \sigma(\alpha)  + z_3 \sigma^2(\alpha) \in \mathcal{O}_F:  z_1 + z_2 + z_3 \equiv 0 \pmod \ell \}.$$
	
	For all $\ell$, the set $M_{\ell}$ is a $\mathbb{Z}$-module. 
	
	We remark that in \cite{DC19}, it is proved that the sublattice $M_{\ell}$ of $\mathcal{O}_F$ has index $\ell$ and it is WR if $\ell\equiv 1\pmod 3$ and $\sqrt{\frac{m}{4}}\le \ell \le \sqrt{4m}$. Thus, if an ideal of $\mathcal{O}_F$ of norm satisfies these conditions,  then that ideal is also WR. We prove the following. 
	\begin{lemma}\label{idealcondition}
		The set $M_{\ell}$ is an ideal of $\mathcal{O}_F$ if and only if $\ell|m$.
	\end{lemma}
	\begin{proof}
		See Appendix \ref{proof_of_lemma_11}.
	\end{proof}
	
	\begin{lemma}\label{idealPi0}
		Assume that $p_i = 3 n_i +1$. Then $p_i \mathcal{O}_F = P_i^3$, where $P_i = \langle p_i, \alpha + n_i \rangle$ is the unique prime ideal above $p_i$. Moreover, one has $-\alpha + \sigma(\alpha) \in P_i$,  $ \|-\alpha + \sigma(\alpha)\|^2= 2m$ and $\| \alpha + n_i\|^2= \frac{2m+p_i^2}{3}$.
	\end{lemma}
	
	\begin{proof}See Appendix \ref{proof_of_lemma_12}.
	\end{proof}
	
	\begin{lemma}\label{traceP}
		We have $M_{p_i}=P_i$. As a consequence, $p_i| \text{Tr}(z)$ for all $z \in P_i$.
	\end{lemma}
	\begin{proof}
		When ${p_i}\mid m$, by Lemma \eqref{idealcondition}, $M_{p_i}$ is an ideal. Moreover, it is a prime ideal above $p_i$ as its index is $p_i$. Therefore $M_{p_i} = P_i$. 
	\end{proof}  
	
	\begin{lemma}\label{lemcomph1}
		Let $m=p_1\cdots p_r \: (r\ge 1)$ and $9\nmid m$. Let $\rho=\alpha -\sigma(\alpha)$. Then $\rho\in P_i$ for all $i = 1,\ldots, r$ and $\|\rho^2\|^2=\text{Tr}(\rho^4) = 2m^2$. \end{lemma}
	\begin{proof}
		By Lemma \ref{traceP}, we have $P_i =\langle p_i, \alpha-n_i\rangle $. The statement $g\in P_i$ is implied from the equalities $\alpha-\sigma(\alpha) = (\alpha-n_i)+(\sigma(\alpha)-n_i)$ and $\sigma(P_i)=P_i,\: \forall i=1,\ldots , r$.
		
		Now, we compute $\|\rho^2\|^2$. First, one has \begin{align}\label{eq:gsquare}
			\|\rho^2\|^2 = (\alpha-\sigma(\alpha))^4+(\sigma(\alpha)-\sigma^2(\alpha))^4+(\sigma^2(\alpha)-\alpha)^4.
		\end{align}
		The right side of \eqref{eq:gsquare} is a symmetric polynomial in the variables 
		\begin{gather*}
			\delta_1= \alpha+\sigma(\alpha)+\sigma^2(\alpha)=1, \\
			\delta_2 = \alpha\sigma(\alpha)+\sigma(\alpha)\sigma^2(\alpha)+\sigma^2(\alpha)\alpha=\frac{1-m}{3} , \\
			\delta_3= \alpha\sigma(\alpha)\sigma^2(\alpha)= \frac{m(a-3)+1}{27}.
		\end{gather*}
		Expressing it in terms of these, one deduces $\|\rho^2\|^2 =  2m^2$.
	\end{proof}

	The following result is new and has not been studied before. We remark that our WR lattice $P$ in Proposition \ref{prop:Psquare_is_WR}  is not one of the sublattices mentioned in \cite[Theorem 4.9]{DM20} since its norm is $m^2$.
	
	\begin{proposition}
		\label{prop:Psquare_is_WR} Let $m =  p_1\cdots p_r $ where $r\ge 1$ and let $P = P_1\cdots P_r$. Then $P^2$ is an orthogonal WR ideal lattice with a minimal basis $\{\kappa,\sigma(\kappa),\sigma^2(\kappa)\}$, where the element $\kappa= m-(\alpha-\sigma(\alpha))^2$.
	\end{proposition}
	
	\begin{proof}
		One has $Tr(\kappa)= m$ and $\|\kappa\|^2 = m^2$. By Lemma \ref{lemindependentcubic}, the set $\{\kappa,\sigma(\kappa),\sigma^2(\kappa)\}$ is $\mathbb{R}-$ linearly independent. It is clear $m\in P^2$ and thus $\kappa\in P^2$ by Lemma \ref{idealPi0}.
		
		Now, we prove $\kappa$ is a shortest vector in $P^2$. First, consider the sublattice of $P^2$ defined as $L=  \ZZ \kappa+\ZZ \sigma(\kappa)+\ZZ \sigma^2(\kappa)\ZZ$. We remark that $\kappa+\sigma(\kappa) = \text{Tr}(\kappa)-\sigma^2(\kappa)= (\alpha-\sigma^2(\alpha))^2$. It leads to 
		\begin{align}\label{eq:eqintheo}
			\|\kappa\|^2 +\|\sigma(\kappa)\|^2+2\text{Tr}(\kappa\sigma(\kappa)) =  \|(\alpha-\sigma^2(\alpha))^2\|=2m^2
		\end{align} 
		by Lemma \ref{lemcomph1}. Since $\|\kappa\|^2 =\|\sigma(\kappa)\|^2 = m^2$, the equality in equation \eqref{eq:eqintheo} implies that $\text{Tr}(\kappa\sigma(\kappa)) =0$. It follows that $\kappa,\sigma(\kappa),\sigma^2(\kappa) $ are pairwise orthogonal. As a consequence, $\det(L) = m^3 =\det(P^2)$ and $\kappa$ is a shortest vector of $L$. By Lemma \ref{sublatticeequal}, one has $L= P^2$. Thus, $P^2$ is an orthogonal WR ideal lattice with a minimal basis $\{\kappa,\sigma(\kappa),\sigma^2(\kappa)\}$.
	\end{proof}

	Let $I$ be a non-empty subset of set $\{1,\ldots,r\}$, $p_I = \prod_{i\in I}p_i$ and $P_I=\prod_{i\in I}P_i$. As a consequence of Lemma \ref{traceP}, $P_I = M_{p_I}$. By \cite[Theorem 4.1]{de2017integral}, if $p_I\in \left[\frac{\sqrt{m}}{2},2\sqrt{m}\right]$ then $P_I$ is WR. Moreover, by using a different technique, independent of the proof of \cite[Theorem 4.1]{de2017integral}, we can prove a stronger result: the condition $p_I\in \left[\frac{\sqrt{m}}{2},2\sqrt{m}\right]$ is not only necessary but also sufficient for $P_I$ to be WR. 
	
	\begin{proposition}\label{prop:cubic9ndividem}
		Let $m = p_1 \cdots p_r$ be the conductor of $F$ and let $P_i$ be the prime ideals above $p_i$ for all $i=1,\ldots, r$. For each nonempty subset $I$ of $\{1,\ldots ,r\}$, let $P_I =\nolinebreak \prod_{i\in I }P_i$, $p_I=\nolinebreak\prod_{i\in I }p_i $ and $n_I = \frac{p_I-1}{3}$. Then $P_I$ is WR if and only if $\frac{m}{4}\le p_I^2 \le 4m$. In this case, $P_I$ has a minimal basis $\alpha+n_I,\sigma(\alpha)+n_I, \sigma^2(\alpha)+n_I$.
	\end{proposition}
	\begin{proof}
		By Lemma \ref{idealPi0}, $P_i = \langle p_i, \alpha+n_i\rangle$ where $n_i = \frac{p_i-1}{3}$ for all $i\in I$. This implies that 
		\[\ZZ\tron{\alpha+n_I}+\ZZ \tron{\sigma(\alpha)+n_I}+\ZZ\tron{\sigma^2(\alpha)+n_I}\subset P_I.\]
		Moreover, this sublattice of $P_I$ and $P_I$ have the same indices in $\mathcal{O}_F$ and thus  \[P_I=\ZZ\tron{\alpha+n_I}+\ZZ \tron{\sigma(\alpha)+n_I}+\ZZ\tron{\sigma^2(\alpha)+n_I}.\]
		Let $\delta$ be a nonzero vector of $P_I$. There exist integers $x_1,x_2,x_3$ such that \[\delta = x_1\tron{\alpha+n_I}+x_2\tron{\sigma(\alpha)+n_I}+x_3\tron{\sigma^2(\alpha)+n_I}.\] %Then \begin{align*}
			%\delta &= \tron{(n_I+1)x_1+n_Ix_2+n_Ix_3}\alpha + \tron{n_Ix_1+\tron{n_I+1}x_2+n_Ix_3}\sigma(\alpha)+\tron{n_Ix_1+n_Ix_2+\tron{n_I+1}x_3}\sigma^2(\alpha).
			%\end{align*}
			By Lemma \ref{lencoeff}, we have $$\|\delta\|^2 = \frac{m}{3}\tron{\tron{x_1-x_2}^2+\tron{x_2-x_3}^2+\tron{x_3-x_1}^2}+\frac{\tron{3n_I+1}^2}{3}\tron{x_1+x_2+x_3}^2.$$ 	
			Now, we will find the minimum value of $\|\delta\|^2$ when $\delta\ne 0$. Note that $z_1z_2z_3\ne 0$. We consider all cases as below.
			\begin{enumerate}[(i)]
				\item If $z_1+z_2+z_3 =0$, then 
				\[(z_1-z_2)^2 +(z_2-z_3 )^2 +(z_3-z_1)^2\ge 2.\] 
				Here $(z_1-z_2)^2 +(z_2-z_3 )^2 +(z_3-z_1)^2$ is an even non-negative integer. If 
				\[(z_1-z_2)^2 +(z_2-z_3 )^2 +(z_3-z_1)^2\in \{2,4\},\] 
				then two of the three numbers $z_1,z_2,z_3$ are zero. Without loss of generality, we can assume $z_1=z_2$. This implies that $z_3 =-2z_1$ and thus 
				$(z_1-z_2)^2 +(z_2-z_3 )^2 +(z_3-z_1)^2$ is a multiple of $9$. Hence, 
				\[(z_1-z_2)^2 +(z_2-z_3 )^2 +(z_3-z_1)^2\ge 6,\] 
				and therefore, $\|\delta\|^2\ge 2m$ in this case. The equality occurs if and only if 
				\[\delta \in \{\pm \tron{\alpha-\sigma(\alpha)},\pm \tron{\sigma(\alpha)-\sigma^2(\alpha)},\pm \tron{\sigma^2(\alpha)-\alpha}\}.\]
				\item If $(z_1-z_2)^2 +(z_2-z_3 )^2 +(z_3-z_1)^2=0$, then $z_1=z_2=z_3=z\in \ZZ$ and thus $\delta =3zp_I$. Hence $\|\delta\|^2 \ge 3p_I^2$. The equality occurs if and only if $\delta \in \{\pm p_I\}$.
				\item If $z_1+z_2+z_3\ne 0 $ and $(z_1-z_2)^2 +(z_2-z_3)^2 +(z_3-z_1)^2\ne 0$, then $(z_1+z_2+z_3)^2 \ge 1 $ and $(z_1-z_2)^2 +(z_2-z_3 )^2 +(z_3-z_1)^2\ge 2$. Thus $\|\delta \|^2\ge \frac{p_I^2+2m}{3}$. The equality occurs if and only if 
				\[\delta \in \{\pm \tron{\alpha+n_I}, \pm \tron{\sigma(\alpha)+n_I},\pm \tron{\sigma^2(\alpha)+n_I}\}.\]
			\end{enumerate}
			Therefore, we conclude that $P_I$ is WR if and only if $\frac{2m+p_I^2}{3}\le \min\{ 2m,3p_I^2\}$, which is equivalent to $ \frac{m}{4} \le\nolinebreak p_I^2 \le\nolinebreak 4m$. 
		\end{proof}
		
		\subsection{The case \texorpdfstring{$9\mid m$}{9|m}}

		Let $m= p_0^2  p_1 \cdots p_r$ where $3 = p_0 < p_1 < p_2 \cdots < p_r$ and $r \ge 0$. For each nonempty subset $I$ of $\{1\cdots ,r\}$, we denote $P_I =  \prod_{i\in I}P_i$. In this section, we will show that: 
		
		\begin{enumerate}[i)]
			\item if $m= 9$, then $P_0$ is WR; 
			\item the ideal $P_0(P_1\cdots P_r)^2$ is orthogonal and WR; 
			\item if $I$ is a nonempty subset of $\{1,2,\cdots,r\}$, then $P_0P_I$ is WR if and only if $\frac{m}{36}\le p_I^2\le \frac{4m}{9}$; and 
			\item if $r\ge 2$ and $I, J$ are two nonempty and disjoint subsets of $\{1,2,\cdots,r\}$, then $P_0P_I^2P_J$ is WR if and only if $\frac{m}{36}\le p_Ip_J^2\le \frac{4m}{9}$. The field $F$ is not tame, and hence is not studied in \cite{DM20} and \cite{DC19}. Indeed, all of our results in this subsection are new and have not been investigated before. 
		\end{enumerate}

		By \cite{maki2006determination}, one has $\{1,\alpha, \sigma(\alpha)\}$ is an integral basis. It can be easily verified that $\alpha$ satisfies $\|\alpha\|^2 = \frac{2m}{3}$ and thus it is a shortest vector in $\mathcal{O}_F\backslash \mathbb{Z}$  (see \cite{TranPeng1} for more details).
		
		\begin{lemma}\label{idealPi}
			Let $m= 9 p_1 \cdots p_r$ where $r \ge 0$. Then $p_i \mathcal{O}_F = P_i^3$ where $P_i$ is the unique prime ideal above $p_i$. Moreover, $P_0=\langle 3, \alpha -1 \rangle$ and $P_i = \langle p_i, \alpha \rangle$ for all $1 \le i \le r$.
		\end{lemma}
		
		\begin{proof}
			To compute generators for $P_i$ we can apply the decomposition of primes  \cite[Theorem 4.8.13]{cohen1993course}  since Lemma \ref{lem:index-cubic} says that $p_i$ does not divide the index $[\mathcal{O}_F: \mathbb{Z}[\alpha]]$. In other words, the result is obtained by factoring the defining polynomial $df(x)$ over the finite field $F_{p_i}$ and by using the fact that $p_i |m$, and $a \equiv 6 \pmod 9$. 
		\end{proof}
		In case $m>9$, using Lemma \ref{idealPi} and the fact that $\frac{2m}{3}> 27 =\|3\|^2$ leads to the following.
		
		\begin{corollary}\label{coridealPi1}
			Let $m >9$. Then the vector $\alpha$ is a shortest vector in the set $P_i\backslash \mathbb{Z}$ for all $ 1 \le i \le r$, and $\|\alpha\|^2 = \frac{2m}{3}$. In the ideal $P_0$, the element $p_0$ is shortest and $\|p_0\|^2 = 27$.
			
		\end{corollary}

		\begin{proposition}
			\label{prop:cubic9dividem2}
			Let $m =9$. Then $P_0$ is orthogonal and WR with a minimal basis $\{\alpha-1, \sigma(\alpha)-1,  \sigma^2(\alpha )-1\}$.
		\end{proposition}

		\begin{proof}
			Note that $\alpha-1 \in P_0$ and this element has trace $-3$ since $\text{Tr}(\alpha)=0$, thus the three elements $\alpha-1$, $\sigma(\alpha)-1$ and  $\sigma^2(\alpha)-1 $ are all in $P_0$ and are linearly independent by Lemma \ref{lemindependentcubic}. To show that $P_0$ is WR, it is sufficient to show that $\alpha-1$ is shortest in $P_0$. 
			
			We have that $\|\alpha-1\|^2= \|\alpha\|^2 + \|-1\|^2 = \frac{2m}{3}+ 3 = 9$ because $\alpha$ has trace $0$. One can easily compute all the shortest vectors of the ideal lattice $P_0$  (see the Fincke--Pohst algorithm -- Algorithm 2.12 in \cite{fincke1985improved}) and verify that $\alpha-1$ is indeed shortest in $P_0$.
		\end{proof}
		
		Lemma \ref{lencoeff} cannot be applied to the case $9\mid m$. Therefore, we recalculate the length of vectors in $\mathcal{O}_F$ in this case as follows.
		\begin{lemma}\label{lem:length-cubic-3divm}
			Let $\delta = m_1 + m_2 \alpha + m_3 \sigma(\alpha) \in \mathcal{O}_F$. Then $$\|\delta\|^2= 3m_1^2+
			\frac{2m}{3}(m_2^2+m_3^2-m_2m_3).$$
		\end{lemma}
		\begin{proof}
			See Appendix \ref{proof_of_lemma_17}.
		\end{proof}
		
		Next, we claim that $P_0P^2$ is an orthogonal and WR lattice where $P= P_1\cdots P_r$ and $r\geq 1$. To prove that, we need some Lemmas below.
		\begin{lemma}
			\label{lem:idealP_3midm}For all $1\le i\le r$, we have $P_i = \ZZ p_i\oplus\ZZ \alpha \oplus\ZZ \sigma(\alpha)$.
		\end{lemma}
		\begin{proof}
			It is clear that $L_i =  \ZZ p_i\oplus\ZZ \alpha \oplus\ZZ \sigma(\alpha)$ is the sublattice of $P_i$ and $\det(L_i) =\det( P_i)$. Therefore $P_i = L_i$ by Lemma \ref{sublatticeequal}.
		\end{proof}
		By using the same technique as in the proof of Lemma \ref{lem:idealP_3midm}, one has the following result.
		\begin{corollary}
			\label{cor:idealP_I} Let $I$ be a subset of $\{1,\ldots,r\}$. Then $P_I =  \ZZ p_I +\ZZ \alpha +\ZZ \sigma(\alpha)$. In particular, $P_1 \cdots P_r =\ZZ \frac{m}{9}\oplus\ZZ \alpha \oplus\ZZ \sigma(\alpha).$
		\end{corollary}
		
		\begin{proposition}\label{prop:cubic9dividem1}
			Let $I$ be a subset of $\{1,\ldots,r\}$. Then $P_I$ is not WR.
		\end{proposition}
		\begin{proof}
			By Corollary \ref{cor:idealP_I}, we have  $P_I =  \ZZ p_I +\ZZ \alpha +\ZZ \sigma(\alpha)$. If we let $\delta \in P_I$, then $\delta = z_1p_I + z_2\alpha +z_3 \sigma(\alpha)$ where $z_1,z_2,z_3\in \ZZ$. By applying Lemma \ref{lem:length-cubic-3divm},  one obtains 
			\[\|\delta\|^2 = 3z_1^2p_I^2+\frac{2m}{3}(z_2^2+z_3^2-z_2z_3).\]
			Now, we will find the minimum value of $\|\delta\|^2$ when $\delta\ne 0$. We consider all cases as below.
			\begin{enumerate}
				\item If $z_1 = 0 $, then $\|\delta\|^2 \ge \frac{2m}{3}$ (since $z_2^2+z_3^2-z_2z_3\ge 1$), here the equality  occurs when $z_2 = 1,z_3=0$ or $z_2=0,z_3=1$, therefore $\delta \in \{\alpha,\sigma(\alpha)\}$.
				\item If $z_1\ne 0$, then $\|\delta\|^2\ge 3z_1^2p_I^2+\frac{2m}{3}(z_2^2+z_3^2-z_2z_3)\ge 3z_1^2p_I^2\ge 3p_I^2$, here the equality  occurs when $z_2=z_3=0,z_1=1$ and thus $\delta = p_I$.
			\end{enumerate}
			In conclusion, $\min_{\delta \ne 0}\|\delta\| \in \left\{\|\alpha\|^2, \|p_I\|^2\right\} = \left\{\frac{2m}{3},3p_I^2\right\}$. Note that $\frac{2m}{3}\ne 3p_I^2$, so in the case $\|p_I\|^2 < \|\alpha\|^2$,  we have $\pm p_I$ are the only two shortest vectors in $P_I$. Therefore, $P_I$ is not WR. In another case $\|p_I\|^2 > \|\alpha\|^2$ and hence $\alpha$ is shortest in $P_I$. 
			
			We will next compute the set of all shortest vectors $L$ of $P_I$. Let $\delta \in \mathcal{O}_F$ such that $\|\delta \|= \|\alpha\|$. Since $\mathcal{O}_F = \mathbb{Z} \oplus \mathbb{Z}[\sigma] \cdot \alpha$ (see \cite[Proposition 2.2 and Proposition 2.3]{TranPeng1}), we can show easily that $\delta\in L=\{\pm \alpha, \pm \sigma(\alpha), \pm \sigma^2(\alpha)\}$.  Moreover, one can observe that $\text{Tr}(\alpha) = \alpha + \sigma(\alpha) + \sigma^2(\alpha)=0$ and $\{\alpha, \sigma(\alpha),\alpha^2(\alpha)\}$ linearly dependent. Therefore, there does not exist three independent vectors from $L$. In other words, $P_I$ is not WR. 
		\end{proof}
		
		\begin{lemma}		\label{lem:exprA_B} 
			There exist integers $A,B$ such that \[A^2 -AB+B^2  = \frac{m}{9}, \text{ and } \alpha^2 =\nolinebreak \frac{2m}{9}+\nolinebreak A\alpha +\nolinebreak B \sigma(\alpha).\]
		\end{lemma}
		\begin{proof}
			See Appendix \ref{proof_of_lemma_19}.
		\end{proof}
		\begin{lemma}
			\label{lem:idealP1Pr} Let $\alpha,A,B$ be in Lemma \ref{lem:exprA_B} and let $\kappa = \frac{m}{9}+A\alpha +B\sigma(\alpha)$. Then $P_0 ( P_1\cdots P_r)^2 =  \ZZ \kappa \oplus\ZZ \sigma(\kappa) \oplus\ZZ\sigma^2(\kappa).$  
		\end{lemma}
		\begin{proof}It is clear that the two lattices $P_0 ( P_1\cdots P_r)$ and $ \ZZ \kappa \oplus\ZZ \sigma(\kappa) \oplus\ZZ\sigma^2(\kappa)$ have the same index in $\mathcal{O}_F$ and thus it is sufficient to prove that $ \ZZ \kappa \oplus\ZZ \sigma(\kappa) \oplus\ZZ\sigma^2(\kappa)$ is a sublattice of $P_0 ( P_1\cdots P_r) $. It is obvious that $\frac{m}{9} \in (P_1 \cdots P_r)^2$. Since $\kappa = \alpha ^2 -\frac{m}{9}$, one has $\kappa \in (P_1\cdots P_r)^2$.
			Moreover, \[\kappa =  \alpha^2 -\frac{m}{9}= (\alpha-1 )(\alpha +1 )+(p_1 \cdots p_r -1 ) \in  P_0 \] as $P_0 =  \langle 3 ,\alpha -1 \rangle $ and $p_1 \equiv 1 \pmod3$. Hence, $\kappa \in  P_0 (P_1\cdots P_r)^2 $. As a consequence, $\sigma (\kappa), \sigma^2 (\kappa)\in  P_0(P_1\cdots P_r)^2.$   
		\end{proof}
		
		Lemma \ref{lem:idealP1Pr} gives us an integral basis of $P_0(P_1\cdots P_r)^2$. Let 
		\[\delta =  z_1 \kappa +z_2 \sigma(\kappa)+z_3\sigma^2(\kappa)\in P_0(P_1\cdots P_r)^2.\]
		One has \begin{align*}
			\delta = \frac{m}{9}( z_1+z_2+z_3)+(Az_1 -Bz_2+(B-A)z_3 )\alpha + (Bz_1+(A-B)z_2-Az_3)\sigma(\alpha).
		\end{align*} 
		We then apply Lemma \ref{lem:length-cubic-3divm}, to obtain that 
		\begin{align}\label{eq:eq_before_propP0Pr_WR}
			\|\delta\|^2 = \frac{m^2 }{27}(z_1+z_2+z_3)^2+\frac{2m}{3}(A^2-AB+B^2)(z_1^2+z_2^2+z_3^2-z_1z_2-z_1z_3-z_2z_3).
		\end{align}

		Since $\frac{m}{9} = A^2 -AB+B^2$, the following result follows.
		\begin{proposition}\label{prop:P0Psquare_WR}
			The ideal $P_0(P_1\cdots P_r)^2$ is orthogonal and WR with a minimal basis 
			$\{\kappa,\sigma(\kappa),\sigma^2(\kappa)\}$ with $\kappa$ as in Lemma \ref{lem:idealP1Pr}.
		\end{proposition}
		
		\begin{proof}
			Let $\delta  \in P_0 (P_1\cdots P_r)^2$. Then there exist integers $z_1,z_2,z_3$ such that we can express $\delta$ as $\delta =z_1 \kappa+z_2 \sigma(\kappa)+z_3 \sigma^2 (\kappa)$ by Lemma \ref{lem:idealP1Pr}. 
			Since $\frac{m}{9} =A^2 -AB+B^2$, the equality in \eqref{eq:eq_before_propP0Pr_WR} implies that \begin{align*}
				\|\delta\|^2 = \frac{m}{9}(z_1^2 +z_2^2 +z_3^2).
			\end{align*}
			When $\delta \ne 0$, it is clear that $\|\delta\|^2 \ge \frac{m^2 }{9}$ as at least one of $z_1,z_2 ,z_3$ is a nonzero integer. Equality holds if and only if $\delta \in \{\pm \kappa,\pm \sigma(\kappa),\pm  \sigma^2(\kappa)\}$. Hence $\{\pm \kappa,\pm \sigma(\kappa),\pm  \sigma^2(\kappa)\}$ is the set of all  shortest vectors of $P_0 (P_1 \cdots P_r)^2.$ Therefore, $P_0(P_1\cdots P_r)^2$ is WR. Moreover, we can verify that $\text{Tr}\tron{\kappa \sigma\tron{\kappa}} = 0$ and thus $P_0(P_1\cdots P_r)^2$ is also orthogonal.
		\end{proof}
		From now on, for each nonempty subset $I$ of $\{ 1,2,\cdots , r\}$, we denote by $p_I = \prod_{i\in I }p_i$ and $P_I   =  \prod_{i\in I } P_i$.
		
		For each $i\in \{ 1,\ldots, r\}$, let $\rho_i = p_i+\alpha +\sigma(\alpha)$. Since 
		\[\rho_i =(p_i-1 )+(\alpha -1)+(\sigma(\alpha)-1) \in P_0\] and clearly $p_i \in P_i$, then $\rho_i\in P_0P_i$. Hence $\ZZ \rho_i +\ZZ\sigma(\rho_i) +\ZZ\sigma^2(\rho_i)$ is a sublattice of $P_0P_i$ and this sublattice has the same determinant as the one of $P_0P_i$. Therefore, we have that $\ZZ \rho_i +\ZZ\sigma(\rho_i) +\ZZ\sigma^2(\rho_i) = P_0P_i$.

		By using the same argument, we can prove the following lemma.
		
		\begin{lemma}
			\label{lem:P0PI_ideal} Let $r\ge 1$ and $I$ be a nonempty subset of $\{ 1,\ldots ,r\}$ and let $ \rho_I = p_I +\alpha +\sigma(\alpha).$ Then $P_I = \ZZ \rho_I  \oplus\ZZ \sigma(\rho_I)\oplus\ZZ\sigma^2 (\rho_I)$. In particular, $P_0P_1\cdots P_r = \ZZ \rho \oplus\ZZ \sigma(\rho)\oplus\ZZ\sigma^2(\rho)$ where  $\rho= \frac{m}{9}+\alpha+\sigma(\alpha)$.  
		\end{lemma}
		
		The following proposition shows the necessary and sufficient conditions for the ideal $P_0P_I^2$ of a given subset $I$ of $\{1,2,\cdots,r\}$ to be a WR lattice.

		\begin{proposition}
			\label{prop:P0PI_WR} Let $I$ be a nonempty subset of $\{ 1,2 \cdots,r\}$. The ideal $P_0P_I$ is WR if and  only if $\frac{m}{36 }\le p_I^2 \le \frac{4m}{9}$. In this case, a minimal basis of $P_0P_I$ is $\{\rho_I,\sigma(\rho_I),\sigma^2(\rho_I)\}$ where  $\rho_I = p_I+\alpha+\sigma(\alpha)$.
		\end{proposition}
		\begin{proof}
			By Lemma \ref{lem:P0PI_ideal}, $P_I =  \ZZ \rho_I \oplus\ZZ \sigma(\rho_I)\oplus\ZZ\sigma^2 (\rho_I)$. Let $\delta =  z_1\rho_I +z_2 \sigma(\rho_I)+z_3 \sigma^2(\rho_I)$. Lemma \ref{lem:length-cubic-3divm} states that\begin{align*}
				\|\delta\|^2 =  3p_I^2 (z_1+z_2+z_3)^2+\frac{m}{3}\left((z_1-z_2)^2+(z_2-z_3)^2+(z_3-z_1)^2\right).
			\end{align*} 
			Now, we will find the minimum value of $\|\delta\|^2$ when $\delta\ne 0$. We consider all cases as below.
			\begin{enumerate}[(i)]
				\item If $z_1+z_2+z_3 =0$, then \[(z_1-z_2)^2 +(z_2-z_3 )^2 +(z_3-z_1)^2\ge 2.\]
				Note that the expression on the left hand side is an even positive integer. If \[(z_1-z_2)^2 +(z_2-z_3 )^2 +(z_3-z_1)^2\in \{2,4\},\] then two of the three numbers $z_1,z_2,z_3$ are zero. Without loss of generality, we can assume $z_1=z_2$. This implies that $z_3 =-2z_1$ and thus the expression is a multiple of $9$. Hence 
				\[(z_1-z_2)^2 +(z_2-z_3 )^2 +(z_3-z_1)^2\ge 6.\]
				Therefore, $\|\delta\|\ge 2m$ in this case. The equality occurs if and only if \[\delta \in \{\pm (\alpha -\sigma(\alpha)),\pm (\alpha-\sigma^2(\alpha)), \pm (\sigma(\alpha)-\sigma^2(\alpha))\}\]
				\item If $(z_1-z_2)^2 +(z_2-z_3 )^2 +(z_3-z_1)^2=0$, then $z_1=z_2=z_3=z\in \ZZ$ and thus $\delta =3zp_I$. Hence $\|\delta\|^2 \ge 27p_I^2$. The equality occurs if and only if $\delta \in \{\pm 3p_I\}$.
				\item If $z_1+z_2+z_3\ne 0 $ and $(z_1-z_2)^2 +(z_2-z_3 )^2 +(z_3-z_1)^2\ne 0$, then $(z_1+z_2+z_3)^2 \ge 1 $ and \[(z_1-z_2)^2 +(z_2-z_3 )^2 +(z_3-z_1)^2\ge 2.\] Thus $\|\delta \|^2\ge 3p_I^2 +\frac{2m}{3}$. The equality occurs if and only if $\delta \in \{\pm g_I, \pm\sigma(g_I),\pm \sigma^2(g_I)\}$.
			\end{enumerate}
			This implies that $\min_{\delta \ne 0}\|g\|^2 = \min\set{2m,27p_I^2, 3p_I^2+\frac{2m}{3}}.$ Since $\text{Tr}(\rho_I)\ne 0$, the ideal $P_0P_I$ is WR if and only if $\min_{\delta \ne 0} \|\delta\|^2 =  3p_I^2 +\frac{2m}{3}$. It is equivalent to the statement $ 3p_I^2+\frac{2m}{3}\le 2m$ and $ 3p_I^2\le 27p_I^2$. These inequalities occur if and only if $\frac{m}{36}\le p_I^2\le \frac{4m}{9}$.
		\end{proof}
		
		Using Proposition \ref{prop:P0PI_WR} for $I =\{1,\cdots,r\}$, we have the following result.
		\begin{corollary}
			\label{cor:P0P1Pr_nowWR} Let $r\ge 1$. Then the ideal $P_0P_1\cdots P_r$ is not WR.
		\end{corollary}
		
		Let $I,J$ be two disjoint nonempty subsets of $\{1,2,\cdots ,r\}$. Now, we show the necessary and sufficient condition for $P_0P_I^2P_J$ to be a WR lattice (Proposition \ref{prop:P0PIsqPj}). 
		\newcommand{\PP}{\mathcal{P}}
		
		Let $\xi_3=\frac{-1-\sqrt{-3}}{2}$ be a primitive cube root of $1$ and $K'=  \QQ(\xi_3)$. The minimal polynomial of $\xi_3$ is $x^2+x+1$. For each $i\in \{1,\ldots ,r\}$, the polynomial $x^2+x+1$ has a root modulo $p_i$. It means $\mathcal{O}_{K'}$ has an ideal $\mathcal{P}_i$ of $\mathcal{O}_{K'}$ of norm $p_i$. For each subset $I$ of $\{1,\ldots,r\}$, let $\PP_I = \prod_{i\in I }\PP_i$. Then $\PP_I$ is an ideal of $\mathcal{O}_{K'}$ norm $p_I$. Moreover, since $\mathcal{O}_{K'}$ is a PID, then there exist integers $x_I,y_I$ such that $\PP_I=\langle x_I+y_I\xi_3 \rangle$ and thus $p_I=\N (x_I+y_I\xi_3)=x_I^2-x_Iy_I+y_I^2$. In other words, the following result has been deduced.
		\begin{lemma}
			\label{lem:exp_pI_xy} For each nonempty subset $I$ of $\{1,\ldots,r\}$, there exist integers $x_I,y_I$ such that $x_I+y_I+1\equiv 0 \pmod 3$ and $		p_I = x_I^2-x_Iy_I+y_I^2.$
		\end{lemma}
		
		\begin{lemma}\label{lem:xI_y_I}
			Let $r\ge 2$ and $N=p_1\cdots p_r$ where $p_i$ is a prime such that $p_i\equiv 1 \pmod 3$ for each $i\in \{1,\ldots ,r\}$. Assume that $N= A^2 -AB+B^2$ where $A,B$ are integers that $A+B+1\equiv 0 \pmod 3$. For each nonempty subset $I$ of $\{ 1,\ldots ,r\}$, let $p_I=  \prod_{i\in I} p_i$. Then there exist integers $x_I,y_I$ such that 
			\begin{align*}
				x_I+y_I+1\equiv 0\pmod 3, & \quad p_I =x_I^2-x_Iy_I+y_I^2 \\
				p_I\mid \tron{Ax_I-By_I-Ay_I}, & \quad p_I\mid (Bx_I-Ay_I).
			\end{align*}
		\end{lemma}
		\begin{proof}
			See Appendix \ref{proof_of_lemma_24}.
		\end{proof}
		
		\begin{lemma}\label{lem:multipleofpIsq}
			Let $N=\frac{m}{9}=p_1\cdots p_r= A^2-AB+B^2$ where $A$ and $B$ as in Lemma \ref{lem:exprA_B}. With the notation in Lemma \ref{lem:xI_y_I}, one has $p_I^2 \mid \N_{K/\QQ}(x_I\alpha+y_I\sigma(\alpha))$. In particular, $x_I\alpha +y_I\sigma(\alpha)\in P_I^2$.
		\end{lemma}
		\begin{proof}
			See Appendix \ref{proof_of_lemma_25}.
		\end{proof}

		\begin{proposition}
			\label{prop:P0PIsqPj}Let $r\ge 2$ and $I,J$ be two disjoint nonempty subsets of $\{ 1,2,\cdots, r\}$. The ideal $P_0P_I^2P_J$ is WR if and only if $\frac{m}{36}\le p_I^2p_J \le \frac{4m}{9}$. In this case, $P_0P_IP_J^2$ has a minimal basis $\{\kappa_{IJ},\sigma(\kappa_{IJ}),\sigma^2(\kappa_{IJ}\}$ where $\kappa_{IJ} = p_{IJ}+x_I+y_I$ and $x_I$ and $y_I$ are given in Lemma \ref{lem:xI_y_I}. 
		\end{proposition}
		\begin{proof}
			With $x_I, y_I$ in Lemma \ref{lem:multipleofpIsq}, one has $x_I\alpha +y_I\sigma(\alpha) \in P_I^2$. By Corollary \ref{cor:idealP_I}, we have $x_I\alpha +y_I\sigma(\alpha)\in P_J$. Thus $\kappa_{IJ} \in P_I^2P_J$ as $I,J$ are disjoint. Moreover, $\kappa_{IJ}\in P_0$ as 
			\begin{gather*}
				\kappa_{IJ}= (p_Ip_J-1)+(\alpha-1)x_I+(\sigma(\alpha)-1)y_I+(x_I+y_I+1)
			\end{gather*}
			and $P_0 = \langle 3,\alpha-1 \rangle, \sigma(P_0)=P_0$ and $3\mid (x_I+y_I+1)$ by Lemma \ref{lem:xI_y_I}. Hence $\kappa_{IJ}\in P_0P_I^2P_J$ and thus $L_{IJ} = \ZZ \kappa_{IJ}\oplus\ZZ \sigma(\kappa_{IJ})\oplus\ZZ\sigma^2(h_{IJ})$ is a sublattice of $P_0P_I^2P_J$. It is easy to verify that $\det(L_{IJ})=\det(P_0P_I^2P_J)$ and thus $L_{IJ}=  P_0P_I^2P_J$ by Lemma \ref{sublatticeequal}.
			
			Let $\delta = z_1 \kappa_{IJ}+z_2\sigma(\kappa_{IJ})+z_3\sigma^2(\kappa_{IJ})$ be a nonzero vector of $P_0P_I^2P_J$. We can write 
			\begin{align*}
				\delta = p_Ip_J\tron{z_1+z_2+z_3}&+\tron{x_Iz_1-y_Iz_2+\tron{y_I-x_I}z_3}\alpha \\ 
				&+\tron{y_Iz_1+\tron{x_I-y_I}z_2-x_Iz_3}\sigma(\alpha)	\end{align*}
			and hence by Lemma \ref{lem:length-cubic-3divm}	
			\begin{align*}
				\|\delta\|^2 =3p_I^2p_J^2\tron{z_1+z_2+z_3}^2+ \frac{2m}{3}\prod_{i\in I}p_i\tron{z_1^2+z_2^2+z_3^2-z_1z_2-z_2z_3-z_1z_3}.
			\end{align*}
			By using a similar argument as the one in the proof of Proposition \ref{prop:P0PI_WR}, one has  \begin{align*}
				\min_{\delta\ne 0}\|\delta\|^2 =\min\set{27p_I^2p_J^2,2mp_I,3p_I^2p_J^2+\frac{2m}{3}p_I},
			\end{align*} 
			and the lattice $P_0P_I^2P_J$ is WR if and only if $\min_{\delta\ne 0} \|\delta\|^2 = 3p_I^2p_J^2 +\frac{2m}{3}p_I$. It is equivalent to the statement 
			\[3p_I^2p_J^2 +\frac{2m}{3}p_I\le 27p_I^2p_J^2 \text{ and } 3p_I^2p_J^2 +\frac{2m}{3}p_I\le 2mp_I.\]
			In other words, $P_0P_I^2P_J$ is WR if and only if  $ \frac{m}{36} \le p_Ip_J^2\le \frac{4m}{9}$.
		\end{proof}
		\section{Well-rounded ideal lattices of cyclic quartic fields}\label{sec:quartic}
		
		In this section, we denote by $F$ a cyclic quartic field defined by $(a,b,c,d)$ as in Section \ref{sec:quarticfields}. We fix the notation where $d= b^2+c^2,$ $\gcd(a,d)=1$ and $a,d$ are squarefree. Let $d= p_1 \cdots p_r$ and $a = sign(a)q_1 \cdots q_s$ be the factorizations of $d$ and $a$ where $sign(a)=1$ if $a>0$ and $sign(a)=-1$ otherwise. Note that all $p_i$ and $q_j$ are distinct since $a$ and $d$ are squarefree and $\gcd(a,d)=1$. For each subset $I$ of $\{1,\ldots,r\}$, let $p_I = \prod_{i\in I}p_i$ and $P_I=\prod_{i\in I}P_I$ where $P_i$ is the unique prime ideal of $\mathcal{O}_F$ above $p_i$ by Lemma \ref{lem:quart_int_bas}, \ref{idealPi0quartic}) for all $i \in I$. In the case $I = \emptyset$, we define $p_I = 1$ and  $P_I = \mathcal{O}_K$. If $J$ is a subset of $\{1,\ldots ,s\}$, we denote $q_J =  \prod_{j\in J}q_j$. Let $J $ be any subset of $\{1, \hdots, s\}$ such that for each $j \in J$, there is a unique prime ideal $Q_j$ above $q_j$. In that case, we denote by $Q_J = \prod_{j\in J}Q_j$.

		\subsection{\texorpdfstring{Prime decomposition of $p\mathcal{O}_F$ and integral bases of ideals of $F$}{Prime decomposition of pOF and integral bases of ideals of F}}\label{sec:int_basis_ideal}

		In this subsection, we provide a number of results concerning the prime factorization of the ideal  $p\mathcal{O}_F$ for an arbitrary prime number $p$. Especially, we aim to classify all odd primes $p$ based on the decomposition of $p\mathcal{O}_F$ (see Theorem \ref{theo:class_p}).  In addition, we construct integral bases for certain ideals of $\mathcal{O}_F$ which can be used to prove the well-roundness of ideals in Section \ref{sec:wr-ideals-quartic}.

		By \cite[Theorem 1.3]{conraddiscriminants}, a prime $p$ ramifies in $F$ if and only if $p\mid \Delta_F$. In this case, by the Lemmas \ref{lem:quart_int_bas} and \ref{lem:quartic_int_basis_divisor_index}, the decomposition of $p\mathcal{O}_F$ is given as below.
		\begin{itemize}
			\item If $p\mid d$, then $p\mathcal{O}_F=  P^4$ where $P $ is a unique prime ideal of $\mathcal{O}_F$ above $p$.
			\item If $p\mid a$ and $d $  is a quadratic non-residue modulo $ p$, then  $p\mathcal{O}_F=  P^2$ where $P $ is a unique prime ideal of $\mathcal{O}_F$ above $p$. 
			\item If $p\mid a$ and $d $  is a quadratic residue modulo $ p$, then  $p\mathcal{O}_F=  P_1^2P_2^2$ where $P_1,P_2 $ are two distinct prime ideals of $\mathcal{O}_F$ above $p$. 
		\end{itemize}
		
		Lemmas  \ref{lem:quart_int_bas} and \ref{lem:quartic_int_basis_divisor_index} show the prime decomposition of $p\mathcal{O}_F$ where $p\mid ac$. Furthermore, if $p\nmid abcd$, then the composition of $p\mathcal{O}_F$ is given as in \ref{lem:quart_int_bas}.\eqref{lem:quart_int_bas_iii}. Eventually, to classify all odd primes $p$, we consider an odd prime divisor $p$ of $b$ such that $p\nmid a$. Lemma \ref{lem:prime_decompose_divisor_b}  is the key component that completes the classification of all odd prime numbers $p$. 
		
		By Lemmas \ref{lem:quart_int_bas}.\eqref{idealPi0quartic} and \ref{lem:quartic_int_basis_divisor_index}, there is a unique prime ideal $P_i$ above a prime $p_i$ for all $p_i\mid d$ and there exists a unique prime ideal $Q_i$ above $q_i$ for all $q_i\mid a$ if $d$ is not a quadratic residue modulo ${q_i}$. We will identify necessary and sufficient conditions for a prime $p$ such that $\mathcal{O}_F$ has a unique prime ideal above $p$ (see Theorem \ref{thm:main6}).

		\begin{remark}\label{rem:belong_PIQJ}
			Let $\delta \in \mathcal{O}_F$. Since $P_i$ is the unique prime ideal above $p_i$, to show that $\delta \in P_I$ it is sufficient to show that $\delta \in P_i$ for all $i \in I$. By Lemma \ref{lem:quartic_int_basis_divisor_index}, $Q_j$ is the unique prime ideal above $q_j$ for all $j\in J$. As consequence, to show $\delta\in Q_J$, it is sufficient to show that $\delta\in Q_j$ for all $j\in J$. Moreover, to claim $\delta\in P_IQ_J$, it is sufficient to show that $\delta\in P_i$ and $\delta\in Q_j$ for all $i\in I, j\in J.$
		\end{remark}

		When ($p\mid d$ or $p\mid a$) and $d$ is a quadratic non-residue modulo $ p$, an integral basis of the unique prime ideal above $p$ is obtained as a consequence of Lemmas \ref{lem:d_even}, \ref{lem:d_odd_b_odd}, \ref{lem:d_odd_b_even_aplusb3}, \ref{lem:d_odd_b_even_ab1_1} and \ref{lem:d_odd_b_even_ab1_2}.
		\begin{lemma}
			\label{lem:d_even} Let $d\equiv 2 \pmod 4$. Then $P_IQ_J = \ZZ p_Iq_J\oplus \ZZ q_J\sqrt{d}\oplus \ZZ \beta \oplus \ZZ \sigma(\beta)$.
		\end{lemma}
		\begin{proof}
			For each $i\in I$ and $j\in J$, since $P_i$ is the unique prime ideal above $p_i$ and $Q_j$ is the unique prime ideal above $q_j$ and by Lemma \ref{lem:norm_some_ele}, one obtains that $\beta,\sigma(\beta)\in P_i$ and $\beta,\sigma(\beta)\in Q_j$. By Remark \ref{rem:belong_PIQJ}, one has $\beta,\sigma(\beta)\in P_IQ_J$. It is obvious to see that $p_Iq_J\in P_IQ_J$ and $q_J\sqrt{d}\in Q_J$. Since $p_I\mid d^2 =\N(\sqrt{d})$ and $P_i$ is the unique prime ideal above $p_i$, we have that $\sqrt{d}\in P_i$ for all $i\in I$ and thus $q_J\sqrt{d}\in P_I$ by Remark \ref{rem:belong_PIQJ}. This means that $q_J\sqrt{d}\in P_IQ_J$. This implies that $L_{IJ} =  \ZZ p_Iq_J\oplus \ZZ q_J\sqrt{d}\oplus \ZZ \beta \oplus \ZZ \sigma(\beta)$ is a sublattice of $P_IQ_J$. However, two lattices $P_IQ_J$ and $L_{IJ}$ have the same indices in $\mathcal{O}_F$. Therefore $P_IQ_J=L_{IJ}$ by Lemma \ref{sublatticeequal}.
		\end{proof}

		\begin{lemma}\label{lem:d_odd_b_odd}
			If $d\equiv 1\pmod 4$ and $b$ is odd, then $P_IQ_J = \ZZ p_Iq_J \oplus \ZZ \frac{q_I\tron{p_I+\sqrt{d}}}{2}\oplus \ZZ \beta \oplus \ZZ\sigma(\beta)$.
		\end{lemma}
		\begin{proof}
			By Remark \ref{rem:integralbasis}, $\frac{1+\sqrt{d}}{2}\in \mathcal{O}_F$ and thus we have $\frac{p_I+\sqrt{d}}{2}=\frac{p_I-1}{2}+\frac{1+\sqrt{d}}{2}\in \mathcal{O}_F$. Since $p_i\mid \tron{\frac{p_I^2-d}{4}}^2=\N\tron{\frac{p_I+\sqrt{d}}{2}}$ and $P_i$ is the unique prime ideal above $p_i$ for all $i\in I$, we have  $\frac{q_J\tron{p_I+\sqrt{d}}}{2}\in P_IQ_J$. By Lemma \ref{lem:norm_some_ele}, $\beta,\sigma(\beta)\in P_i, Q_j$ for all $i\in I$ and $j\in J$. By Remark \ref{rem:belong_PIQJ}, we obtain $\beta,\sigma(\beta)\in P_IQ_J$. One can prove the result using a similar argument as in the proof of Lemma \ref{lem:d_even}.
		\end{proof}

		\begin{lemma}
			\label{lem:d_odd_b_even_aplusb3}Let $d\equiv 1 \pmod 4$, $b$ be even and $a+b\equiv 3\pmod 4$. Then 
			\[P_IQ_J = \ZZ p_Iq_J\oplus \ZZ \frac{q_J\tron{p_I+\sqrt{d}}}{2}\oplus \ZZ \frac{\beta+\sigma(\beta)}{2}\oplus \ZZ\frac{-\beta+\sigma(\beta)}{2}.\]
		\end{lemma}
		Next, we consider the case $d\equiv 1 \pmod 4$, $b \equiv 0\pmod 2$, $a+b\equiv 1 \pmod 4$ and $a\equiv-c\pmod 4$. Let $\gamma'_1, \gamma'_2, \gamma'_3,\gamma'_4$ be a integral basis of $\mathcal{O}_F$ as in Remark \ref{rem:integralbasis}.\ref{rem:integralbasis4}. We define \begin{align}\label{eq:basis_aplusb1_1}
			\gamma_1= \gamma_1', \: \gamma_2= \gamma_2',\: \gamma_3 = -\gamma_4',\: \gamma_4 = \gamma_2'-\gamma_3'.
		\end{align} It is obvious to see that $\{\gamma_1,\gamma_2,\gamma_3,\gamma_4\}$ is a basis of $\mathcal{O}_F$ by \ref{rem:integralbasis}.\ref{rem:integralbasis4}. One has the following result.
		\begin{lemma}
			\label{lem:d_odd_b_even_ab1_1}Let $d\equiv 1 \pmod 4,$ $b$ be even, $a+b\equiv 1\pmod 4$ and $a\equiv -c\pmod 4$. Then \[P_IQ_J = \ZZ \rho_{IJ}\oplus \ZZ \sigma\tron{\rho_{IJ}}\oplus \ZZ\sigma^2\tron{\rho_{IJ}}\ZZ\sigma^3\tron{\rho_{IJ}},\] where $\rho_{IJ}= \frac{-p_Iq_J+q_J\sqrt{d}-\beta-\sigma(\beta)}{4}.$
		\end{lemma}
		
		\begin{proof}
			By Remark \ref{rem:belong_PIQJ}, it is sufficient to prove $\rho_{IJ}\in P_i,Q_J$ for all $i\in I$ and $j\in J$. Let $\gamma_1,\gamma_2,\gamma_3,\gamma_4$ as in \eqref{eq:basis_aplusb1_1}. Then 
			\[\rho_{IJ}=  \frac{-p_Iq_J-q_J+2}{4}\gamma_1+\frac{q_J-1}{2}\gamma_2+\gamma_4\] and thus $\rho_{IJ}\in \mathcal{O}_F$. Moreover, 
			\[\N\tron{\rho_{IJ}} = \frac{\tron{p_Iq_J^2+q_I^2d-2ad}^2-2d\tron{p_Iq_J^2+|a|c}^2}{256}\] and thus $\rho_{IJ}\in P_i,Q_j$ for all $i\in I$ and $j\in J$. As a result 
			\[\rho_{IJ},\: \sigma\tron{\rho_{IJ}},\:\sigma\tron{\rho_{IJ}},\:\sigma^3\tron{\rho_{IJ}}\in P_IQ_J.\]
			Hence $L_{IJ}= \ZZ \rho_{IJ}\oplus \ZZ \sigma\tron{\rho_{IJ}}\oplus \ZZ\sigma^2\tron{\rho_{IJ}}\ZZ\sigma^3\tron{\rho_{IJ}}$ is a sublattice of $P_IQ_J$. Two lattices $P_IQ_J$ and $L_{IJ}$ have the same indices $p_Iq_J^2$ in $\mathcal{O}_F$. Therefore $P_IQ_J=L_{IJ}$.
		\end{proof}
		
		In the remaining case where $d \equiv 1 \pmod 4,b\equiv 0\pmod 2, a+b\equiv 1 \pmod 4$ and $a\equiv c\pmod 4$, using a similar technique to the one in the proof of Lemma \ref{lem:d_odd_b_even_ab1_1}, one obtains the result as below.
		\begin{lemma}
			\label{lem:d_odd_b_even_ab1_2}Let $d \equiv 1 \pmod 4,\: b\equiv 0\pmod 2,\: a+b\equiv 1 \pmod 4$ and $a\equiv c\pmod 4$. Then \[
			P_IQ_J = \ZZ \rho_{IJ}\oplus \ZZ \sigma\tron{\rho_{IJ}}\oplus \ZZ\sigma^2\tron{\rho_{IJ}}\ZZ\sigma^3\tron{\rho_{IJ}},\] where $\rho_{IJ} =  \dfrac{p_Iq_J-q_J\sqrt{d}-\beta+\sigma(\beta)}{4}.$
		\end{lemma}
		
		Next, we will describe a prime ideal above $q_i$ where $q_i\mid a$ and $d$ is a quadratic residue modulo ${q_i}$. By Lemma \ref{lem:quartic_int_basis_divisor_index}, there exist exactly two prime ideals above $q_i$. Let $z_1,z_2$ be two positive integers such that $z_i^2\equiv d\pmod {q_j}$. By the result on the decomposition of primes \cite[Theorem 4.8.13]{cohen1993course}, one has $q_j\mathcal{O}_K =  \mathfrak{q}_{1j}\mathfrak{q}_{2j}$, where $K=[\sqrt{d}]$.

		Before proceeding, we will outline a strategy to prove that a certain lattice is an ideal in Lemmas \ref{lem:q_i_d_even} to \ref{lem:prideals_above_2_2}. The proofs can be seen in the Appendix \ref{appendix_B}.
		
		\begin{remark}\label{stra:strategy_to_prove_ideal}
			Let $\mathcal{O}_F = \ZZ \gamma_1'\oplus \ZZ \gamma_2'\oplus\ZZ \gamma_3'\oplus \ZZ\gamma_4'$, where the $\gamma_i'$ are as in Remark \ref{rem:integralbasis} and let $L = \ZZ \delta_1\oplus \ZZ \delta_2 \oplus  \ZZ \delta_3\oplus \ZZ\delta_4$ where each $\delta_i \in \mathcal{O}_F$. To prove $L$ is an ideal of $\mathcal{O}_F$, we will show that $\delta_i\gamma_j' \in L$ for all $i, j$. In other words,  we perform the following steps for all $ 1 \le i, j  \le 4$.
			\begin{enumerate}[(1)]
				\item Compute $\delta_i\gamma_j'$.
				\item Express $\delta_i\gamma_j' = z_1\delta_1'+z_2\delta_2' +z_3 \delta_3'+z_4\delta_4'$.
				\item Prove that all numbers  $z_1,z_2,z_3,z_4$ are integers.
			\end{enumerate}
		\end{remark}

		When $d$ is even, $df_K(x) = x^2-d$ is a defining polynomial of $K=\QQ(\sqrt{d})$.  Then \begin{align}\label{eq:z_1_z_2}
			df_K(x) \equiv \tron{x-z_i}\tron{x-z_2}\pmod {p_j}.
		\end{align} By using the result on the  decomposition of primes in  \cite[Theorem 4.8.13]{cohen1993course}, one has $\mathfrak{q}_{kj}=\nolinebreak  \ZZ q_j\oplus \ZZ\tron{z_i+\sqrt{d}}$. With  $z_1,z_2$ as in \eqref{eq:z_1_z_2}, one has the result as follows. \begin{lemma}
			\label{lem:q_i_d_even} If $d$ is even and $q_i\mid a$ such that $d$ is a quadratic residue modulo $ {q_j}$, then there exist exactly two prime ideals $Q_{1j},Q_{2j}$ above $q_j$ where \begin{align*}
				Q_{kj}&= \ZZ q_j\oplus \ZZ \tron{z_k+\sqrt{d}}\oplus \ZZ \beta\oplus \ZZ \sigma(\beta).
			\end{align*}
		\end{lemma}
		
		When $d$ is odd, $df_K(x)=  x^2-x+\frac{1-d}{4}$ is a defining polynomial of $K$. One has \begin{align*}
			4df_K(x)\equiv  (2x-1)^2-d \equiv \tron{2x-1-z_1}\tron{2x-1-z_2}\pmod {q_j}.
		\end{align*} As $q\equiv 1 \pmod 4$, there exist integers $t_1,t_2$ such that $z_k = 4t_k-1\pmod{ q_j}$ for $k=1,2,$  and thus $df_K(x) \equiv \tron{x-t_1}\tron{x-t_2}\pmod {q_j}$.

		\begin{lemma}\label{lem:q_i_db_odd} 
			If $d\equiv 1\pmod 4, \: b\equiv 1\pmod 2$ and $q_i\mid a$ such that $d$ is a quadratic residue modulo $ {q_j}$, then there exist exactly two prime ideals $Q_{1i},Q_{2i}$ above $q_i$ where $k=1,2$ and 
			\begin{align*}
				Q_{kj}&= \ZZ q_j\oplus \ZZ \tron{\frac{4t_k-1+\sqrt{d}}{2}}\oplus \ZZ \beta\oplus \ZZ \sigma(\beta).
			\end{align*}
		\end{lemma}

		\begin{lemma}	\label{lem:q_i_d_odd_b_even_ab3} If $d\equiv 1\pmod 4,\: b\equiv 0\pmod 2, \: a+b\equiv 3\pmod 4$ and $q_i\mid a$ such that $d$ is a quadratic residue modulo $ {q_j}$, then there exist exactly two prime ideals $Q_{1i},Q_{2i}$ above $q_i$ where $k=1,2$ and 
			\begin{align*}
				Q_{kj}= \ZZ q_j &\oplus \ZZ \tron{\frac{4t_k-1+\sqrt{d}}{2}} \\ & \qquad\oplus \ZZ \frac{\beta+\sigma(\beta)}{2} \oplus \ZZ \frac{\beta-\sigma(\beta)}{2}.
			\end{align*}
		\end{lemma}

		\begin{lemma}\label{lem:qnotquadratic_ab_1_mod_4_1}
			If $d\equiv 1\pmod 4, \:b\equiv   0 \pmod 2,\: a +b \equiv 1 \pmod 4$ and $a\equiv -c\pmod 4$ and $p_j\mid a$ such that $d$ is a quadratic residue modulo $ {q_j}$, then there are exactly two prime ideals $Q_{1j}, Q_{2j}$ above $q_j$ such that \begin{align*}
				Q_{kj}= \ZZ q_j  \oplus \ZZ \frac{4t_k-1+\sqrt{d}}{2} &  \oplus \ZZ \frac{4t_k-1 +\sqrt{d}-\beta- \sigma(\beta)}{4}\\ & \qquad\oplus \ZZ \frac{2q_j+4t_k-1+\sqrt{d} +\beta-\sigma(\beta)}{4}.
			\end{align*}
		\end{lemma}
		
		\begin{lemma}\label{lem:qnotquadratic_ab_1_mod_4_2}
			If $d\equiv 1\pmod 4, \: b\equiv   0 \pmod 2,\: a +b \equiv 1 \pmod 4$ and $a\equiv c\pmod 4$ and $p_j\mid a$ such that $d$ is a quadratic residue modulo $ {q_j}$, then there exist integers $t_1,t_2$ and exactly two prime ideals $Q_{1j}, Q_{2j}$ above $q_j$ such that $q_j\nmid t_1-t_2$, $d \equiv (4t_i-1)^2\pmod {q_j}$ and \begin{align*}
				Q_{ij}= \ZZ q_j\oplus \ZZ \frac{4t_i-1+\sqrt{d}}{2}&\oplus \ZZ \frac{4t_i-1+2q_j +\sqrt{d}-\beta- \sigma(\beta)}{4} \\ 
				& \qquad\oplus \ZZ \frac{4t_i-1+\sqrt{d} +\beta-\sigma(\beta)}{4}.
			\end{align*}
			
		\end{lemma}
		
		Now, consider a prime $p$ such that $p\mid b$ and $p\nmid a$, Lemma \ref{lem:quartic_int_basis_divisor_b} does not provide us the exact prime decomposition of $p\mathcal{O}_F$. To see this decomposition, it is sufficient to show that $\mathcal{O}_F$ has either a prime ideal of norm $p^2$ or a prime ideal of norm $p.$
		
		\begin{lemma}	\label{lem:prime_decompose_divisor_b}
			Let $p\mid b$ and $p\nmid a$. One has the following.
			\begin{enumerate}[i)]
				\item Assume $2 \mid d$. If $a$ is a quadratic non-residue modulo $ p$, then $p\mathcal{O}_F= P_1P_2$ where \begin{align*}
					P_1 &= \ZZ  p \oplus \ZZ \tron{c+\sqrt{d}}\oplus \ZZ p\sigma(\beta)\oplus \ZZ\tron{\beta+\sigma(\beta)}, \text{ and} \\
					P_2&= \ZZ  p \oplus \ZZ \tron{-c+\sqrt{d}}\oplus \ZZ p\sigma(\beta)\oplus \ZZ\tron{\beta-\sigma(\beta)}
				\end{align*}
				are all prime ideals of $\mathcal{O}_F$ above $p$. 
				
				If $a$ is a quadratic residue modulo $ p$ and we write $a \equiv l^2 \pmod p$, then we have $p\mathcal{O}_F =P_1P_2P_3P_4$ where \begin{align*}
					P_1& =  \ZZ p \oplus \ZZ \tron{c+\sqrt{d}}\oplus \ZZ  \tron{lc-\sigma(\beta)}\oplus \ZZ \tron{lc+\beta}, \\
					P_2& =  \ZZ p \oplus \ZZ \tron{c+\sqrt{d}}\oplus \ZZ  \tron{lc+\sigma(\beta)}\oplus \ZZ \tron{-lc+\beta}, \\
					P_3& = \ZZ p \oplus\ZZ \tron{-c+\sqrt{d}}\oplus \ZZ \tron{lc-\sigma(\beta)}\oplus \ZZ \tron{lc-\beta}, \text{ and}\\
					P_4 & =\ZZ p \oplus\ZZ \tron{-c+\sqrt{d}}\oplus \ZZ \tron{lc+\sigma(\beta)}\oplus \ZZ \tron{lc+\beta} 
				\end{align*} 
				are all prime ideals  of $\mathcal{O}_F$ above $p$.
				
				\item  Assume $d\equiv 1\pmod 4$ and $b\equiv 1 \pmod 2$. If $ a$ is a quadratic non-residue modulo $ p$, then $p\mathcal{O}_F= P_1P_2$ where 
				\begin{align*}
					P_1 &= \ZZ p \oplus \ZZ \frac{p+c+\sqrt{d}}{2} \oplus \ZZ p\sigma\tron{\beta}\oplus\ZZ \tron{\beta+\sigma(\beta)}, \text{ and}\\
					P_2 &= \ZZ p \oplus \ZZ \frac{p-c+\sqrt{d}}{2} \oplus \ZZ p\sigma\tron{\beta}\oplus\ZZ \tron{\beta-\sigma(\beta)} 
				\end{align*} 
				are all prime ideals of $\mathcal{O}_F$ above $p$.  
				
				If $a$ is a quadratic residue modulo $ p$ and we write $a \equiv l^2 \pmod p$, then we have $p\mathcal{O}_F =P_1P_2P_3P_4$ where\begin{align*}
					P_1& =  \ZZ p \oplus \ZZ \frac{p-c+\sqrt{d}}{2}\oplus \ZZ  \tron{lc-\sigma(\beta)}\oplus \ZZ \tron{lc+\beta}, \\
					P_2& =  \ZZ p \oplus \ZZ \frac{p-c+\sqrt{d}}{2}\oplus \ZZ  \tron{lc+\sigma(\beta)}\oplus \ZZ \tron{-lc+\beta}, \\
					P_3& = \ZZ p \oplus\ZZ \frac{p+c+\sqrt{d}}{2}\oplus \ZZ \tron{lc-\sigma(\beta)}\oplus \ZZ \tron{lc-\beta}, \text{ and}\\
					P_4 & =\ZZ p \oplus\ZZ \frac{p+c+\sqrt{d}}{2}\oplus \ZZ \tron{lc+\sigma(\beta)}\oplus \ZZ \tron{lc+\beta} 
				\end{align*}
				are all prime ideals of $\mathcal{O}_F$ above $p$.
				
				\item Assume $d \equiv 1 \pmod 4, b\equiv 0\pmod 2$ and $a+b\equiv 3\pmod 4$. If $a$ is a quadratic non-residue modulo $ p$, then $p\mathcal{O}_F = P_1P_2$ where \begin{align*}
					P_1 &=\ZZ p\oplus \ZZ\frac{-c+\sqrt{d}}{2}\oplus \ZZ \frac{\sigma(\beta)-\beta}{2}\oplus \ZZ p\frac{\beta+\sigma(\beta)}{2}, \text{ and}\\
					P_2 &=\ZZ p\oplus \ZZ\frac{c+\sqrt{d}}{2}\oplus \ZZ p\frac{\sigma(\beta)-\beta}{2}\oplus \ZZ \frac{\beta+\sigma(\beta)}{2}
				\end{align*} 
				are all prime ideals above $p$. 
				
				If $a$ is a quadratic residue modulo $ p$ and we write $a \equiv l^2 \pmod p$, then we have $p\mathcal{O}_F =P_1P_2P_3P_4$ where 
				\begin{align*}
					P_1 & = \ZZ p\oplus\ZZ \frac{-c+\sqrt{d}}{2}\oplus \ZZ \frac{\sigma(\beta)-\beta}{2}\oplus \ZZ\tron{lc-\frac{\beta+\sigma(\beta)}{2}}, \\
					P_2 & = \ZZ p\oplus\ZZ \frac{-c+\sqrt{d}}{2}\oplus \ZZ \frac{\sigma(\beta)-\beta}{2}\oplus \ZZ\tron{lc+\frac{\beta+\sigma(\beta)}{2}}, \\
					P_3 & = \ZZ p\oplus\ZZ \frac{c+\sqrt{d}}{2}\oplus \ZZ \tron{lc+\frac{\sigma(\beta)-\beta}{2}}\oplus \ZZ\frac{\beta+\sigma(\beta)}{2}, \text{ and}\\
					P_4 & = \ZZ p\oplus\ZZ \frac{c+\sqrt{d}}{2}\oplus \ZZ \tron{lc-\frac{\sigma(\beta)-\beta}{2}}\oplus \ZZ\frac{\beta+\sigma(\beta)}{2}				
				\end{align*}
				are all primes ideals of $\mathcal{O}_F$ above $p$.
				
				\item Assume $d\equiv 1 \pmod 4, b \equiv 2\pmod 4, a+b\equiv 1\pmod 4$ and $a\equiv -c \pmod 4$. If $a$ is a quadratic non-residue modulo $ p$, then $p\mathcal{O}_F=P_1P_2$ where 
				\begin{align*}
					P_1 &=  \ZZ p\oplus \ZZ \frac{-c+\sqrt{d}}{2}\oplus \ZZ \frac{b-c+\sqrt{d}-\beta+\sigma(\beta)}{4}\oplus \ZZ\frac{-p+p\sqrt{d}+p\beta+p\sigma(\beta)}{4}, \text{ and}\\
					P_2 &=  \ZZ p\oplus \ZZ \frac{c+\sqrt{d}}{2}\oplus \ZZ \frac{p+p\sqrt{d}-p\beta+p\sigma(\beta)}{4}\oplus \ZZ\frac{b-c-\sqrt{d}-\beta-\sigma(\beta)}{4}
				\end{align*} 
				are all prime ideals of $\mathcal{O}_F$  above $p$.  
				
				If $a$ is a quadratic residue modulo $ p$ and we write $a \equiv l^2\pmod p$, then we have $p\mathcal{O}_F =  P_1P_2P_3P_4$ where 
				\begin{align*}
					P_1 & =  \ZZ p \oplus \ZZ \frac{c+\sqrt{d}}{2}\oplus\ZZ \frac{\tron{-2l+1}c+\sqrt{d}-\beta+\sigma(\beta)}{4}\oplus \ZZ\frac{b-c-\sqrt{d}-\beta-\sigma(\beta)}{4}, \\
					P_2 & = \ZZ p \oplus \ZZ \frac{-c+\sqrt{d}}{2}\oplus\ZZ\frac{b-c+\sqrt{d}-\beta+\sigma(\beta)}{4}\oplus \ZZ\frac{\tron{2l+1}c-\sqrt{d}-\beta-\sigma(\beta)}{4}, \\
					P_3 & = \ZZ p \oplus \ZZ \frac{c+\sqrt{d}}{2}\oplus\ZZ\frac{\tron{2l+1}c+\sqrt{d}-\beta+\sigma(\beta)}{4}\oplus \ZZ\frac{b-c-\sqrt{d}-\beta-\sigma(\beta)}{4}, \text{ and}\\
					P_4 & = \ZZ p \oplus \ZZ \frac{-c+\sqrt{d}}{2}\oplus\ZZ\frac{b-c+\sqrt{d}-\beta+\sigma(\beta)}{4}\oplus \ZZ\frac{\tron{-2l+1}c-\sqrt{d}-\beta-\sigma(\beta)}{4}
				\end{align*}
				are all prime ideals of $\mathcal{O}_F$ above $p$.
				
				\item Assume $d\equiv 1 \pmod 4, b \equiv 2\pmod 4, a+b\equiv 1\pmod 4$ and $a\equiv c \pmod 4$. If $a$ is a quadratic non-residue modulo $ p$, then $p\mathcal{O}_F=P_1P_2$ where 
				\begin{align*}
					P_1 &=  \ZZ p\oplus \ZZ \frac{-c+\sqrt{d}}{2}\oplus \ZZ \frac{b-c+\sqrt{d}-\beta+\sigma(\beta)}{4}\oplus \ZZ\frac{p+p\sqrt{d}+p\beta+p\sigma(\beta)}{4}, \text{ and}\\
					P_2 &=  \ZZ p\oplus \ZZ \frac{c+\sqrt{d}}{2}\oplus \ZZ \frac{-p+p\sqrt{d}-p\beta+p\sigma(\beta)}{4}\oplus \ZZ\frac{b-c-\sqrt{d}-\beta-\sigma(\beta)}{4}
				\end{align*} 
				are all prime ideals of $\mathcal{O}_F$  above $p$.  If $a$ is a quadratic residue modulo $ p$ and we write $a = l^2\pmod p$, then $p\mathcal{O}_F =  P_1P_2P_3P_4$ where 
				{\small \begin{align*}
						P_1 & =  \ZZ p \oplus \ZZ \frac{c+\sqrt{d}}{2}\oplus\ZZ \frac{\tron{2l+1}c+\sqrt{d}-\beta+\sigma(\beta)}{4}\oplus \ZZ\frac{b-c-\sqrt{d}-\beta-\sigma(\beta)}{4}, \\
						P_2 & = \ZZ p \oplus \ZZ \frac{-c+\sqrt{d}}{2}\oplus\ZZ\frac{b-c+\sqrt{d}-\beta+\sigma(\beta)}{4}\oplus \ZZ\frac{\tron{2l+1}c-\sqrt{d}-\beta-\sigma(\beta)}{4}, \\
						P_3 & = \ZZ p \oplus \ZZ \frac{c+\sqrt{d}}{2}\oplus\ZZ\frac{\tron{-2l+1}c+\sqrt{d}-\beta+\sigma(\beta)}{4}\oplus \ZZ\frac{b-c-\sqrt{d}-\beta-\sigma(\beta)}{4}, \text{ and}\\
						P_4 & = \ZZ p \oplus \ZZ \frac{-c+\sqrt{d}}{2}\oplus\ZZ\frac{b-c+\sqrt{d}-\beta+\sigma(\beta)}{4}\oplus \ZZ\frac{\tron{-2l+1}c-\sqrt{d}-\beta-\sigma(\beta)}{4}
				\end{align*}}
				are all prime ideals of $\mathcal{O}_F$ above $p$.
			\end{enumerate}
		\end{lemma} 
		
		\begin{proof}
			The given lattices are completely distinct, and we can prove that they are ideals by following the steps in Remark \ref{stra:strategy_to_prove_ideal}.
		\end{proof}
		
		Finally, we consider prime ideals above $2$ when $\Delta_F$ is even. The following result is obtained from Lemma \ref{lem:quart_int_bas}.\eqref{idealPi0quartic}.
		
		\begin{lemma}\label{lem:ideal_above_2_dodd}
			Let $d$ be even. Then there exists a unique prime ideal $P_0$ above $p_0=2$. Moreover, $P_0 = \langle 2, \beta \rangle$ and $\N\tron{P_0}=2$.
		\end{lemma}
		
		\begin{lemma}\label{lem:ideal_above_2_bodd}
			Let $d\equiv 1 \pmod 4$ and $b$ be odd.
			\begin{enumerate}[(i)]
				\item If $d\equiv 5\pmod 8$, then there is a unique prime ideal $P_0$ above $p_0=2$, where $\N(P_0)=4$ and \begin{align*}
					P_0 = \ZZ 2\oplus \ZZ(1+\sqrt{d})\oplus \ZZ\beta \oplus \ZZ \sigma(\beta).
				\end{align*}
				
				\item If $d\equiv 1 \pmod 8$, then there are exactly two distinct prime ideals $P_{01}, P_{02}$ above $p_0=2$, where $\N(P_{01})=  \N\tron{P_{02}}=2$ and \begin{align*}
					P_{01}&= \ZZ 2\oplus \ZZ\tron{\frac{-1+\sqrt{d}}{2}}\oplus \ZZ \beta \oplus \ZZ\sigma(\beta), \text{ and }\\  
					P_{02}&= \ZZ 2\oplus \ZZ\tron{\frac{1+\sqrt{d}}{2}}\oplus \ZZ \beta \oplus \ZZ\sigma(\beta).
				\end{align*}
			\end{enumerate}
		\end{lemma} 
		
		\begin{lemma}\label{lem:prideals_above_2_2}
			Let $d\equiv 1 \pmod 4$ and $b\equiv 0 \pmod 2$ and $a+b\equiv 3 \pmod 4$.\begin{enumerate}[(i)]
				\item If $d\equiv 5\pmod 8$, then there is a unique prime ideal $P_0$ above $p_0=2$, where $\N(P_0)=4$ and \begin{align*}
					P_0 = \ZZ 2\oplus \ZZ(1+\sqrt{d})\oplus \ZZ\frac{-1+\sqrt{d}-\beta-\sigma(\beta)}{2} \oplus \ZZ \frac{1+\sqrt{d}+\beta-\sigma(\beta)}{2}.
				\end{align*}
				\item If $d\equiv 1 \pmod 8$, then there are exactly two prime ideals $P_{01}, P_{02}$ above $p_0=2$, where $\N(P_{01})=  \N\tron{P_{02}}=2$ and \begin{align*}
					P_{01}&= \ZZ 2\oplus \ZZ\tron{\frac{-1+\sqrt{d}}{2}}\oplus \ZZ \frac{2-\beta-\sigma(\beta)}{2} \oplus \ZZ\frac{\beta-\sigma(\beta)}{2}, \text{ and }\\  
					P_{02}&= \ZZ 2\oplus \ZZ\tron{\frac{1+\sqrt{d}}{2}}\oplus \ZZ \frac{\beta +\sigma(\beta)}{2} \oplus \ZZ\frac{2+\beta -\sigma(\beta)}{2}.
				\end{align*}
			\end{enumerate}
		\end{lemma} 
		For the case of $p=2$ and $\Delta_F$ odd, we have the following result.
		
		\begin{lemma}	\label{lem:p2_delta_odd}
			Assume that $d\equiv 1 \pmod 4$ and $a+b\equiv 1 \pmod 4$. 
			\begin{enumerate}[i)]
				\item If $d \equiv 1\pmod 8$, then $2\mathcal{O}_F$ can be factored as one of the forms $ P_1P_2$, and $P_1P_2P_3P_4$ where $P_1,P_2,P_3,P_4$ are  prime ideals of $\mathcal{O}_F$ above $2$.
				\item\label{lem:p2_delta_odd2} If $d\equiv 5 \pmod 8$, then $2\mathcal{O}_F$ is prime.  
			\end{enumerate}
		\end{lemma}
		\begin{proof}
			\begin{enumerate}[i)]
				\item This is deduced directly from the fact that $p\mathcal{O}_K$ splits totally in $\mathcal{O}_K$ where $\mathcal{O}_K$ as in \eqref{eq:exten_tower}.
				\item See Appendix \ref{proof_lem_p2}.
		\end{enumerate}\end{proof}
		The below theorem follows directly from the combination of Lemmas \ref{lem:quart_int_bas}, \ref{lem:quartic_int_basis_divisor_index},\ref{lem:quartic_int_basis_divisor_b} and \ref{lem:prime_decompose_divisor_b}.

		\begin{theorem}\label{theo:class_p}
			Let $F$ be a cyclic quartic field defined by $a,b,c,d$ as in \eqref{df-polynomial-cubic} and $p$ be an odd prime. One has the following statements.
			\begin{enumerate}[i)]
				\item The prime $p$ is totally ramified if and only if $p\mid d$.
				\item The ideal $p\mathcal{O}_F$ is of the forms $p\mathcal{O}_F = P^2$ for $P$ a unique prime ideal of $\mathcal{O}_F$ above $p$ if and only if $p\mid a$ and $d$ is a quadratic non-residue modulo $ p$.
				\item The ideal $p\mathcal{O}_F$ is of the form $p\mathcal{O}_F = P_1^2P_2^2$ where $P_1,P_2$ are exactly two prime ideals of $\mathcal{O}_F$ above $p$ if and only if $p\mid a$ and $d$ is a quadratic residue modulo $ p$.
				\item The prime $p$ is inert  if and only if $p\nmid abcd$ and $d$ is a quadratic non-residue modulo $ p$.
				
				\item The prime $p$ totally splits if and only if $p$ satisfies one of the conditions listed below.
				\begin{itemize}
					\item The prime $p\mid b$ and $a$ is a quadratic residue modulo $ p$.
					\item The prime $p\mid c$ and $2a$ is a quadratic residue modulo $ p$.
					\item The prime $p\nmid abcd$, $d$ is a quadratic residue modulo $ p$, and if $d\equiv z^2 \pmod p$ then $ad+abz$ and $ad-abz$ are also quadratic residues modulo $ p$.	\end{itemize}
				
				\item The ideal $p\mathcal{O}_F$ is the product of two distinct prime ideals in all the remaining cases.
			\end{enumerate}
		\end{theorem}

		From Theorem \ref{theo:class_p} and Lemmas \ref{lem:ideal_above_2_dodd}, \ref{lem:ideal_above_2_bodd}, and \ref{lem:prideals_above_2_2}, we obtain the necessary and sufficient conditions for a prime $p$ for which  $\mathcal{O}_F$ has a unique prime ideal $P$ over $p$. In the next subsection, we will investigate the conditions for the unique prime ideals $P$ mentioned above to be WR.

		%%%%%%%%%%%%%%%%%%%%%%%%%%%%%%%%%%%%%%%%%%%%%%%%%%%%%%%%%%%%%%
		
		\subsection{Well-rounded ideals of cyclic quartic fields}\label{sec:wr-ideals-quartic}

		According to the first part of Theorem \ref{thm:main6}, there are three cases in which   $\mathcal{O}_F$ has a unique prime ideal $P$ over a prime number $p$. However, in the last case of the theorem, $P=p\mathcal{O}_F$ and it is not primitive. Therefore, we only investigate prime ideals $P$ belonging to the first two cases of the theorem. In general, we will prove necessary and sufficient conditions for an ideal of the form $P_IQ_J$ to be WR, where $I$ is a subset of $\{1,\ldots,r\}$ and $J$ is a subset of $\{1,\ldots, s\}$ such that $d$ is a non-quadratic residue modulo $q_j$ for all $j\in J$.
		\begin{proposition}
			\label{prop:PIQ_J_d_even_notWR} Let $d\equiv 2 \pmod 4$. Then $P_IQ_J$ is not WR.
		\end{proposition}
		\begin{proof}
			Let $\delta\in P_IQ_J$ be a nonzero vector of $P_IQ_J$. By Lemma \ref{lem:d_even}, there exist integers $x_1,x_2,x_3,x_4$ such that $\delta  = x_1 p_Iq_J+x_2q_J\sqrt{d}+x_3 \beta+x_4\sigma(\beta)$ and by \eqref{length}, one obtains
			\begin{align*}
				\|\delta \|^2 = 4\tron{x_1^2p_I^2q_J^2 +x_2^2q_J^2d+|a|d\tron{x_3^2+x_4^2}}.
			\end{align*}
			It is easy to verify that $\min_{\delta\ne 0}\|\delta\|^2\in \min \mathcal{S}$, where $\mathcal{S} =\set{4p_I^2q_J^2, 4q_J^2d, 4|a|d}.$ Each value in $\mathcal{S}$ corresponds to the squared length of at most two independent vectors. Thus, $P_IQ_J$ is not WR.
		\end{proof}
		%Second, we consider the case $d\equiv 1 \pmod 4$ and $ b$ is odd.
		\begin{proposition}
			\label{prop:bd_odd_notWR} Let $d\equiv 1\pmod 4$ and $b$ be odd. Then $P_IQ_J$ is not WR.
		\end{proposition}
		\begin{proof}
			Let $\delta\in P_IQ_J$ be a nonzero vector of $P_IQ_J$. By Lemma \ref{lem:d_even}, there exist integers $x_1,x_2,x_3,x_4$ such that $\delta  = x_1 p_Iq_J+x_2q_J\frac{p_I+\sqrt{d}}{2}+x_3 \beta+x_4\sigma(\beta)$ and by \eqref{length}, one obtains
			\begin{align*}
				\|\delta \|^2 = \tron{2x_1+x_2}^2p_I^2q_J^2 +x_2^2q_J^2d+4|a|d\tron{x_3^2+x_4^2}.
			\end{align*}
			Since $2x_1+x_2$ and $x_2$ have the same parity, it is easy to verify that $\min_{\delta\ne 0}\|\delta\|^2\in \min \mathcal{S}$, where $\mathcal{S} =\set{p_I^2q_J^2+q_J^2d, 4|a|d}.$ Each value in $\mathcal{S}$ corresponds to the squared length of at most two independent vectors. Thus $P_IQ_J$ is not WR.
		\end{proof}
		
		\begin{proposition}\label{prop:dodd_b_ven_aplusb3_mod4}
			Let $d\equiv 1 \pmod 4$, $b$ be even and $a+b\equiv 3\pmod 4$. Then $P_IQ_J $ is not WR. 
		\end{proposition}
		\begin{proof}
			Let $\delta\in P_IQ_J$ be a nonzero vector of $P_IQ_J$. By Lemma \ref{lem:d_even}, there exist integers $x_1,x_2,x_3,x_4$ such that $\delta  = x_1 p_Iq_J+x_2q_J\frac{p_I+\sqrt{d}}{2}+x_3 \frac{\beta+\sigma(\beta)}{2}+x_4\frac{-\beta+\sigma(\beta)}{2}$ and by \eqref{length}, one obtains
			\begin{align*}
				\|\delta \|^2 = \tron{2x_1+x_2}^2p_I^2q_J^2 +x_2^2q_J^2d+2|a|d\tron{x_3^2+x_4^2}.
			\end{align*}
			The result is then obtained using the same argument as in the proof of Proposition \ref{prop:bd_odd_notWR}.\end{proof}

		\begin{proposition}\label{lem:WR_aplusb1_1}Suppose that $d\equiv 1 \pmod 4,\:\: b\equiv 0\pmod 2,\:\: a+b\equiv1\pmod 4$ and $a\equiv -c\pmod 4$. Then $P_IQ_J$ is WR if and only if $p_I^2q_J^2+q_J^2d+2|a|d\le \min \mathcal{M}$, where $$\mathcal{M}=\set{ 16q_J^2d, 8|a|d,  4q_I^2d+4|a|d, 16p_I^2q_J^2,4p_I^2q_J^2+4|a|d,4p_I^2q_J^2+4q_J^2d}.$$	 
		\end{proposition}
		\begin{proof}
			Let $\rho_{IJ}$ be in Lemma \ref{lem:d_odd_b_even_ab1_1} and $\delta $ be a nonzero vector of $P_IQ_J$. By Lemma \ref{lem:d_odd_b_even_ab1_1}, there exist integers $x_1,x_2,x_3,x_4$ such that $4\delta =  S_1p_Iq_J+S_2q_J\sqrt{d}+S_3\beta +S_4\sigma(\beta)$ where 
			\begin{align*}
				S_1= -x_1-x_2-x_3-x_4, \quad & S_2 = x_1-x_2+x_3-x_4,  \\S_3 = -x_1+x_2+x_3-x_4,  \quad &  S_4 =  -x_1-x_2+x_3+x_4 .
			\end{align*}
			By \eqref{length}, one has \begin{align*}
				4\|\delta\|^2 =  S_1^2 p_I^2q_J^2 +S_2^2q_J^2d+|a|d\tron{S_3^2+S_4^2}. 
			\end{align*} 
			% \HT{$S_1^2$?}
			It is easy to prove that $\min_{\delta\ne 0}\|4\delta\|^2 =\min \mathcal{S}$ where
			$$\mathcal{S} = \set{p_I^2q_J^2+q_I^2d+2|a|d,\:  16q_J^2d, \: 8|a|d, 4q_I^2d+4|a|d,  \:16p_I^2q_J^2, 4p_I^2q_J^2+4|a|d, \:4p_I^2q_J^2+4q_I^2d}.$$
			Among seven numbers in $\mathcal{S}$, the only one that is correspondent to the squared length of four linearly independent vectors in $P_I$ is $ p_I^2q_J^2+q_J^2d+2|a|d$. Therefore, the lattice $P_IQ_J$ is WR if and only if $\min_{\delta \ne 0}4\|\delta\|^2 = p_I^2q_J^2+q_J^2d+2|a|d$.
		\end{proof}

		\begin{proposition}	\label{lem:WR_aplusb1_2} Suppose that $d\equiv 1 \pmod 4, b\equiv 0\pmod 2, a+b\equiv1\pmod 4$ and $a\equiv c\pmod 4$. Then $P_IQ_J$ is WR if and only if $p_I^2q_J^2+q_J^2d+2|a|d\le \min \mathcal{M}$ where $$\mathcal{M}=\set{ 16q_J^2d,  8|a|d,  4q_J^2d+4|a|d, 16p_I^2q_J^2, 4p_I^2q_J^2+4|a|d, 4p_I^2q_J^2+4q_J^2d}.$$	 
		\end{proposition}
		\begin{proof}
			Let $\rho_{IJ}$ be in Lemma \ref{lem:d_odd_b_even_ab1_2} and $\delta $ be a nonzero vector of $P_IQ_J$. By Lemma \ref{lem:d_odd_b_even_ab1_2}, there exist integers $x_1,x_2,x_3,x_4$ such that $4\delta =  S_1p_Iq_J+S_2q_J\sqrt{d}+S_3\beta +S_4\sigma(\beta)$ where $S_1= x_1+x_2+x_3+x_4,S_2 = -x_1+x_2-x_3+x_4,S_3 = -x_1-x_2+x_3+x_4, S_4 =  x_1-x_2-x_3+x_4$. By \eqref{length}, one has \begin{align*}
				4\|\delta\|^2 =  S_1^2 p_I^2q_J^2 +S_2^2q_J^2d+|a|d\tron{S_3^2+S_4^2}. 
			\end{align*} 
			% \HT{$S_1^2$?}
			It is not hard to verify that $\min_{\delta\ne 0}\|4\delta\|^2 =\min \mathcal{S}$ where $$\mathcal{S} = \set{p_I^2q_J^2+q_I^2d+2|a|d,  16q_J^2d,  8|a|d,  4q_I^2d+4|a|d, 16p_I^2q_J^2, 4p_I^2q_J^2+4|a|d, 4p_I^2q_J^2+4q_I^2d}.$$
			Among seven numbers in $\mathcal{S}$, the only one that corresponds to the squared length of four linearly independent vectors in $P_I$ is $ p_I^2q_J^2+q_J^2d+2|a|d$. Therefore, the lattice $P_IQ_J$ is WR if and only if $\min_{\delta \ne 0}4\|\delta\|^2 = p_I^2q_J^2+q_J^2d+2|a|d$.
		\end{proof}
		We now prove Theorem \ref{thm:main_5}.
		
		\begin{proof}[ Proof of Theorem \ref{thm:main_5}]
			\begin{enumerate}[i)]
				\item 	By Propositions \ref{prop:PIQ_J_d_even_notWR}, \ref{prop:bd_odd_notWR}, \ref{prop:dodd_b_ven_aplusb3_mod4}, \ref{lem:WR_aplusb1_1} and \ref{lem:WR_aplusb1_2}, the ideal $P_I$ is WR if and only if  $d\equiv 1\pmod 4$, $p\equiv 0\pmod 2$, $ a+b\equiv 1\pmod 4$ and $p_I^2+\nolinebreak\tron{2|a|+1}d\le \min \mathcal{S}$, where 
				$$\mathcal{S} = \set{16d, 8|a|d, \: 4d+4|a|d,\: 16p_I^2,\: 4p_I^2+4|a|d,4p_I^2+4d}.$$ 
				The last inequality is equivalent to the statement 
				\begin{align*}
					&p_I^2+\tron{2|a|+1}d\le 16d,\\ 
					&p_I^2+\tron{2|a|+1}d\le 4d+4|a|d, \\
					&p_I^2+\tron{2|a|+1}d\le 16p_I^2, \\ 
					&p_I^2+\tron{2|a|+1}d\le 4p_I^2+4d.
				\end{align*}
				This means \begin{align}
					\label{eq:ineq_WR_PI}\max \left\{\frac{(2|a|-3)d}{3},\frac{(2|a|+1)d}{15}\right\}\le p_I^2\le \min\left\{ \tron{15-2|a|}d,\tron{2|a|+3}d\right\}.
				\end{align} 
				The inequalities in \eqref{eq:ineq_WR_PI} occur only if $2|a|\le 15$ and thus $|a|\in \{1,3,5,7\}$. \begin{itemize}
					\item If $|a|=1$, the inequalities in \eqref{eq:ineq_WR_PI} become $\frac{d}{5}\le p_I^2\le 5d$.
					\item If $|a|=3$, the inequalities in \eqref{eq:ineq_WR_PI} become $d\le p_I^2\le 9d$.
					\item If $|a|=5$, the inequalities in \eqref{eq:ineq_WR_PI} become $\frac{7d}{3}\le p_I^2\le 5d$.
					\item If $|a|=7$, the inequalities in \eqref{eq:ineq_WR_PI} lead to $\frac{11d}{3}\le p_I^2\le d$, which is impossible.
				\end{itemize} 
				\item By Propositions \ref{prop:PIQ_J_d_even_notWR}, \ref{prop:bd_odd_notWR}, \ref{prop:dodd_b_ven_aplusb3_mod4}, \ref{lem:WR_aplusb1_1} and \ref{lem:WR_aplusb1_2}, one can show that $Q_J$ is WR if and only if $d\equiv 1\pmod 4,$ $ p\equiv 0\pmod 2,$ $ a+b\equiv 1\pmod 4$ and $q_J^2(d+1)+2|a|d\le \min \mathcal{S}$, where 
				$$\mathcal{S} = \set{ 16q_J^2d,\: 8|a|d, \: 4q_J^2d+4|a|d,\: 16q_J^2, \: 4q_J^2+4|a|d,\: 4q_J^2+4q_J^2d}.$$
				The last inequality is equivalent to \begin{align*}
					&q_J^2(d+1)+2|a|d\le 16q_J^2,\\
					&q_J^2(d+1)+2|a|d\le 8|a|d,\\& q_J^2(d+1)+2|a|d\le 16q_J^2,\\ &q_J^2(d+1)+2|a|d\le 4q_J^2+4d.
				\end{align*} 
				This means $d<15$ and \begin{align}
					\label{eq:ineq_WR_QJ} \max\set{\frac{2|a|d}{15-d},\frac{2|a|d}{3\tron{d+1}}}\le q_J^2\le \min\left\{ \frac{6|a|d}{d+1},\frac{2|a|d}{d-3}\right\}.
				\end{align} Since $d$ is odd and squarefree, and $d<15$, one must have $d\in\{5,13\}.$ If $d=13$ then \eqref{eq:ineq_WR_QJ} becomes $13|a|\le q_J^2\le\frac{13|a|}{5} $, which is impossible. Thus $d$ must be $5$ and the inequalities in \eqref{eq:ineq_WR_QJ} become 
				$|a|\le q_J^2\le 5|a|$.
			\end{enumerate}
			
		\end{proof}
		Now we consider prime ideals above $2$.
		\begin{lemma}\label{lem:pideal_above_2_not_WR}
			No prime ideal above $2$ is WR if $d$ is even or if $d$ is odd and $b\equiv 1\pmod 2$.
		\end{lemma}
		\begin{proof}
			When $d$ is even, the result is directly implied from Proposition \ref{prop:PIQ_J_d_even_notWR}. The result in the remaining case can be obtained by using a similar argument to the proofs of Propositions \ref{prop:PIQ_J_d_even_notWR} and \ref{prop:bd_odd_notWR}.
		\end{proof}
		By employing the same methodology used to prove Propositions \ref{prop:PIQ_J_d_even_notWR} and \ref{prop:bd_odd_notWR}, we can establish the result of Lemma \ref{lem:2_ab3_not_WR}.
		
		\begin{lemma}\label{lem:2_ab3_not_WR}
			Let $d\equiv 1 \pmod 8, b\equiv 0\pmod 2$ and $a+b\equiv 3\pmod 4$. Then all prime ideals above $2$ are not WR.
		\end{lemma}
		\begin{lemma}\label{lem:ideal_2_WR}
			Let $d\equiv 5 \pmod 8, b\equiv 0\pmod 2$ and $a+b\equiv 3\pmod 4$. Then $\mathcal{O}_F$ has a unique prime ideal $P_0$ above $2$. Moreover, $P_0$ is WR if and only if $a =1,b=2,c=1,d=5$.
		\end{lemma}
		\begin{proof}
			By Lemma \ref{lem:prideals_above_2_2}, there is a unique prime ideal $P_0$ above $2$ and an integral basis of $P_0$ is given as in this lemma. Let $0\ne \delta\in P_0$, there are integers $z_1,z_2,z_3,z_4$ such that 
			\[\delta =2z_1+z_2\tron{1+\sqrt{d}}+z_3 \frac{-1+\sqrt{d}-\beta-\sigma(\beta)}{2}+z_4 \frac{1+\sqrt{d}+\beta-\sigma(\beta)}{2}\]
			and by \eqref{length}, one obtains $\|\delta\|^2 = S_1 ^2 +S_2^2d+|a|d\tron{S_3^2+S_4^2},$
			where 
			\begin{align*}
				&S_1=  4x_1+2z_2-z_3+z_4, \\
				&S_2= 2z_2+z_3+z_4,\\
				&S_3= -z_3+z_4,\\
				&S_4= -z_3-z_4. 
			\end{align*}It is easy to prove that $\min_{\delta \ne 0}\|\delta\|^2 =\min\set{16,1+d\tron{2|a|+1}}$ and $P_0$ is WR if and only if $16\ge 1+\tron{2|a|+1}$. It occurs only if $a=1,b=2,c=1$.   
		\end{proof}
		
		\begin{remark}\label{rem:el_WR}
			If $P$ is a ideal above $2$, then $2\in P$. Thus, if $P$ is WR, then there exists $\delta \in P\setminus \QQ(\sqrt{d})$ such that $\|\delta\|^2\le 16$.
		\end{remark}
		\begin{lemma}\label{lem:ideal_2_notWR}
			Let	$d\equiv 1 \pmod 4, b\equiv 0\pmod 2$ and $a+b\equiv 1\pmod 4$. Then all prime ideals above $2$ are not WR.
		\end{lemma} 
		\begin{proof}
			If $d\equiv 5 \pmod 8$, then $2\mathcal{O}_F$ is prime (see Lemma \ref{lem:p2_delta_odd}) and not primitive. We now consider the case $d\equiv 1 \pmod 8$ here. Note that $d\ge 17 $ as $d\equiv 1 \pmod 4$ and $d$ is squarefree. We divide into two sub-cases: $a\equiv -c\pmod 4$ and $a\equiv c\pmod 4$. Since the techniques used in the proofs of the two cases are similar, we only consider the first. In this case, suppose that there exists a prime ideal $P$ above $2$ such that $P$ is WR. Hence, by Remark \ref{rem:el_WR}, there exists $\delta \in P\setminus \QQ(\sqrt{d})$ such that $\|\delta\|^2\le 16$. Let $\gamma_1',\gamma_2',\gamma_3',\gamma_4'$ be as in Remark \ref{rem:integralbasis}. There exist integers $z_1,z_2,z_3,z_4$ such that $\delta =  z_1\gamma_1'+z_2\gamma_2'+z_3\gamma_3'+z_4\gamma_4'$ and thus \begin{align*}
				\|\delta\|^2=\frac{1}{4}\tron{\tron{4z_1+2z_2+z_3+z_4}^2+d\tron{2z_2+z_3-z_4}+2|a|d\tron{z_3^2+z_4^2}}.
			\end{align*}
			
			Since $\delta\notin \QQ(\sqrt{d})$, one has $z_3^2+z_4^2\ge 1$. Hence, $|a|d \le \|\delta\|^2 \le 32 $ which occurs only if $|a|=1$ and $d\le 32$. This means $(a,d)\in \set{(1,17),(-1,17)}$ as $d\equiv 1 \pmod 8$ and $d$ is squarefree. In both cases of $(a,d)$, there are two prime ideals above $2$ and we can verify that these prime ideals are not WR by using Pari/GP. Hence, all prime ideals above $2$ are not WR when $d\equiv 1 \pmod 4, b\equiv 0\pmod 2$ and $a+b\equiv 1\pmod 4$.  
		\end{proof}
		Combining Lemmas \ref{lem:pideal_above_2_not_WR}, \ref{lem:2_ab3_not_WR}, \ref{lem:ideal_2_WR} and \ref{lem:ideal_2_notWR}, we imply Proposition \ref{prop:ideal_2_WR}.

		\begin{proposition}	\label{prop:ideal_2_WR}Let $F,a,b,c,d$ be as in 
			Section \ref{sec:quarticfields}. Then a prime ideal above $2$ of $\mathcal{O}_F$ is WR if and only if $a=1,b=2,c=1,d=5$. In this case, $\mathcal{O}_F$ has a unique prime ideal above $2$.
		\end{proposition} 
		
		\section{Conclusion and future research}\label{sec:conclusion}
		
		This paper investigates WR ideals of cyclic and quartic fields. We show that all cyclic cubic fields have WR ideals. 
		Moreover, we present families of cyclic cubic and quartic fields of which WR ideal lattices exist and also construct explicit minimal bases of these WR ideals. 
		
		We observe that all WR ideals obtained from our experiment have norms dividing the discriminant of the field if the discriminant is odd.  Therefore, we form the following conjecture.
		
		\textbf{Conjecture:} Let $F$ be a  cyclic cubic or cyclic quartic field with an odd discriminant. If a primitive integral ideal $I$ of $F$ is WR, then $N(I)$  divides the discriminant of $F$.  
		
		If this conjecture holds then there are only finitely many WR ideals from each of these fields.
		
		Note that this conjecture agrees with the observation in \cite{FHLPSW13} for real quadratic fields, and it was later proved for these fields \cite{S20}. In addition, for a cyclic quartic field $F$ of odd discriminant, the conjecture holds for the case when the ideal $I$ of $F$ is the unique prime ideal above a prime number as a result of Theorem \ref{thm:main6}.
		
		We also remark that the conjecture does not hold for cyclic quartic fields of even discriminant. That is, there exist cyclic quartic fields with even discriminant which have WR ideals of norms that do not divide the field discriminant. For example, the cyclic quartic field $F$ defined by $(a,b,c,d)=(1,2,1,5)$ has WR ideals with norms 484, 2420, 3364, and 3844 which do not divide $\Delta_F=2000$. Another remark is that this is the only case in which a prime ideal above $2$ is WR by Proposition \ref{prop:ideal_2_WR}.
		
		Our future research will investigate the above conjecture and WR ideals of other number fields.
		
		\subsection*{Acknowledgements}
		The authors would like to thank Amy Feaver for her help to improve the initial version of this manuscript and to thank the reviewer for their constructive comments that helped improve the manuscript.
		Ha T. N. Tran was supported by the Natural Sciences and Engineering Research Council of Canada (NSERC) (funding RGPIN-2019-04209 and DGECR-2019-00428).
		
		%%% REFERENCES %%%
		{\small\bibliography{dat}}
		% Please, do not change the above line and do not insert your references
		% into this file.  Instead, insert your references into the commat.bib file.
		% See commat.bib for further instructions.
		
		%%Appendix
		\appendix
		\section{Some results related to cyclic cubic fields}
		\begin{proof} [Proof of Lemma \ref{lem:length-cubic-3divm}]\label{proof_of_lemma_17}
			Recall that $\Tr(\alpha) = \alpha + \sigma(\alpha)+\sigma^2(\alpha) = 0$. 	We have
			\begin{align*}
				\sigma^2(\delta) &= m_1-m_3\sigma(\alpha)+ (m_2-m_3)(-\alpha-\sigma(\alpha)) \\&=m_1+(m_3-m_2)\alpha - m_2\sigma(\alpha).
			\end{align*}
			Thus 
			\begin{align*}
				\|\delta\|^2 &=\delta^2 +\sigma(\delta)^2+(\sigma^2(\delta))^2 \\ &= 3m_1^2 + 2(m_2^2+m_3^2-m_2m_3)(\alpha^2+\sigma(\alpha)^2+\alpha\sigma(\alpha))\\&= 3m_1^2 +\frac{2m}{3}(m_2^2+m_3^2-m_2m_3).
			\end{align*}
			The last equality occurs because of the fact that 
			\begin{align*}
				\alpha^2+\sigma(\alpha)^2+\alpha\sigma(\alpha) 	&= -\alpha\sigma(\alpha) + (\alpha+\sigma(\alpha))^2 \\
				&= -\alpha\sigma(\alpha) - (\alpha+\sigma(\alpha))\sigma^2(\alpha) \\&= \frac{m}{3}.
			\end{align*}
		\end{proof}
		\begin{proof}[Proof of Lemma \ref{lem:exp_pI_xy}]\label{proof_of_lemma_19}
			Let $P = P_1 \cdots P_r$. From Corollary \ref{cor:idealP_I} and from $\alpha ^2 \in P^2$, there exists integers $k,A,B$ such that $\alpha^2 = k\frac{m}{9}+A \alpha +B\sigma(\alpha )$. The value of $k$ is $2$ since $\text{Tr}(\alpha )=\text{Tr}(\sigma(\alpha)) = 0$ and $\text{Tr}(\alpha^2)=\frac{2m}{3}$. By using Lemma \ref{lem:length-cubic-3divm}, one deduces that 
			\[\|\alpha^2\| ^2 =  \frac{4m^2}{29}+\frac{2m}{3}(A^2-AB+B^2).\]
			It is easy to show that $\|\alpha\|^2 =  \frac{2m^2}{9}$. Therefore $A^2 -AB+B^2 = \frac{m}{9}$.
		\end{proof}
		\begin{proof}[Proof of Lemma \ref{idealcondition}]\label{proof_of_lemma_11}
			By using the coefficients of the defining polynomial of $F$ in \eqref{df-polynomial-cubic}, one has 
			\begin{equation}\label{eqtrace}
				\text{Tr}(\alpha) = \text{Tr}(\sigma(\alpha))=1, \text{Tr}(\alpha^2)= \frac{2m+1}{3}  \text{  and  } \text{Tr}(\alpha \sigma(\alpha)) = \frac{1-m}{3}.
			\end{equation}
			Note that the  set $M_\ell $ can be defined equivalently as $M_\ell=\{\delta\in \mathcal{O}_F:\text{Tr} (\delta)\equiv 0\pmod{\ell}\}$. Let $\delta=a_1\alpha+a_2\sigma(\alpha)+a_3\sigma^2(\alpha)\in M_\ell$. Then $a_1+a_2+a_3\equiv 0\pmod\ell$. By computation, we obtain \begin{align*}
				\text{Tr}(\delta\alpha)&=\frac{1-m}{3}(a_1+a_2+a_3)+a_1m\\
				\text{Tr}(\delta\sigma(\alpha))&=\frac{1-m}{3}(a_1+a_2+a_3)+a_2m\\
				\text{Tr}(\delta\sigma^2(\alpha))&=\frac{1-m}{3}(a_1+a_2+a_3)+a_3m.
			\end{align*} 
			If $\ell \mid m$, then $\text{Tr}(\delta\alpha)=\text{Tr}(\delta\sigma(\alpha))=\text{Tr}(\delta\sigma^2(\alpha))\equiv 0\pmod\ell$. Thus, all of $\delta\alpha$, $\delta\sigma(\alpha)$ and  $\delta\sigma^2(\alpha)$ are in $I$. Since $\{\alpha, \sigma(\alpha), \sigma^2(\alpha)\}$ is a basis of $\mathcal{O}_F$ (Lemma \ref{integralbasis-3notdiv9}), one has that  $M_\ell$ is an ideal. 
			
			Conversely, assume that $M_\ell$ is ideal. 
			Then the element  $\alpha-\sigma(\alpha)$ has trace $0$ and hence is in  $ M_{\ell}$.  Thus $\alpha(\alpha- \sigma(\alpha)) \in M_{\ell}$ since $\alpha \in \mathcal{O}_F$ and  $M_{\ell}$ is an ideal. Therefore,  by \eqref{eqtrace}, $\text{Tr}(\alpha(\alpha- \sigma(\alpha)))= \text{Tr}(\alpha^2) - \text{Tr}(\alpha \sigma(\alpha)) = m\equiv 0 \pmod \ell$. In other words, $\ell|m$.
		\end{proof}
		\begin{proof}[Proof of Lemma \ref{idealPi0}] \label{proof_of_lemma_12}
			By  using the fact that $p_i |m$ and $n_i = -3^{-1} \pmod {p_i}$, one can factor $df(x)$ as $df(x) \equiv (\alpha + n_i)^3 \pmod {p_i}$. On the other hand, Lemma \ref{lem:index-cubic} says that $p_i$ does not divide the index $[\mathcal{O}_F: \mathbb{Z}[\alpha]]$.  Therefore, one has $P_i = \langle p_i, \alpha + n_i \rangle$ by using the result on the  decomposition of primes  \cite[Theorem 4.8.13]{cohen1993course}.
			
			First, $-\alpha + \sigma(\alpha) = -(\alpha+n) + (\sigma(\alpha)+n) \in P_i$ since we have proved that $P_i = \langle p_i, \alpha + n_i \rangle$ and by the fact that $\sigma(P_i) =P_i $. The length of this element is easily computed by applying Lemma \ref{lencoeff}.
			
			Next, we compute the length of $\alpha + n_i$. By writing 
			\[\alpha + n_i = \alpha + n_i(\alpha + \sigma(\alpha) + \sigma^2(\alpha) )= (n_i+1) \alpha + n_i \sigma(\alpha) + n_i\sigma^2(\alpha)\] and applying Lemma \ref{lencoeff}, the result is obtained.
		\end{proof}
		\begin{proof}
			[Proof of Lemma \ref{lem:xI_y_I}]
			\label{proof_of_lemma_24}Let $F=  \Q(\xi_3)$ and $\theta = A+B\xi_3$. Then $\N(\theta)= N$ and there exists an ideal $\PP_1,\cdots \PP_i$ such that $\N(\PP_i)=p_i$ and 
			\[\theta \mathcal{O}_K = \PP_1\cdots \PP_r = \prod_{i\in I}\PP_i \prod_{j\notin I}\PP_j.\] Since $\mathcal{O}$ is a PID, then there exist elements $x_j+y_j\xi_3\in \mathcal{O}_K$ and $x_I+y_I\xi_3 \in\mathcal{O}_K$ such that $x_j+y_j\equiv 1\pmod 3$ for all $j\notin I,x_I+y_I\equiv 1\pmod 3$ and 
			$\PP_i = \langle \delta_i \rangle, \PP_I=\langle \delta_I\rangle$ whereas $\delta_i = x_i+y_i\xi_3$ and $\delta_I = x_I +y_I\xi_3$. It leads to the equality $\theta\mathcal{O}_K = \left(\delta_I\prod_{j\notin I}\delta_j\right)\mathcal{O}_K$ and thus there exists $\varepsilon\in \mathcal{O}_K^*$ such that $\theta\varepsilon =  \delta_I\prod_{j\notin I}\delta_j$. 
			Let $\sigma_F(\delta_I)$ be the conjugate of $\delta_I$ over $F$. One has $\delta_I\sigma_F(\delta_I) = p_I$ and thus 
			\[\theta \varepsilon \sigma_F\tron{\delta_I} = \tron{\prod_{j\notin I}\delta_j}\tron{\delta_I\sigma_F\tron{\delta_I}} =  p_I\tron{\prod_{j\ne i}\delta_j}.\] It means $\theta\sigma_F(\delta_I)\in p_I\mathcal{O}_K$. Moreover, \[\theta \sigma_F(\delta_I)= Ax_I+By_I-Ay_I+(Bx_I-Ay_I)\xi_3\]
			and thus $Bx_I-Ay_I,Ax_I+By_I-Ay_I$ are multiples of $p_I$.\end{proof}
		\begin{proof}
			[Proof of Lemma \ref{lem:multipleofpIsq}]
			\label{proof_of_lemma_25}Let $\gamma = x_I\alpha+y_I\sigma(\alpha).$ Remark that $\frac{m}{9}=A^2-AB+B^2$ and $ \alpha^2 = \frac{2m}{9}+A\alpha +B\sigma(\alpha)$. Since $\text{Tr}(\alpha \sigma(\alpha ))=-\frac{n}{3}$, then we can write $ \alpha \sigma(\alpha ) =\frac{-m}{9}+C\alpha +D\sigma\alpha$ for some integers $C,D$. One has $\alpha^3 = \frac{m\alpha}{3}+\frac{am}{27}$ and $\alpha^3 = \frac{2m\alpha}{9}+A\alpha^2+B\alpha\sigma\alpha$. This implies that  \begin{align*}
				\frac{m\alpha}{3}+\frac{am}{27}= \tron{AB+BD}\sigma(\alpha)+\tron{\frac{2m}{9}+A^2+BC}\alpha +\tron{\frac{2mA}{9}-\frac{mB}{9}}
			\end{align*}
			and thus $AB+BD =0, \frac{2m}{9}+A^2+BC =\frac{m}{3},\frac{ma}{27} = \frac{2mA}{9}-\frac{mB}{9}$. Since $B$ must be nonzero, $A=-D$ and it is easy to prove $C=B-A$.
			
			One can easily verify that \begin{align*}
				%\gamma &=  x_I\alpha+y_I\sigma(\alpha)\\
				\gamma \alpha &= \frac{m}{9}\tron{2x_I-y_I} +\tron{Ax_I+By_I-Ay_I}\alpha+\tron{Bx_I-Ay_I}\sigma(\alpha),\\
				\gamma\sigma(\alpha)&= \frac{m}{9}\tron{-x_I+2y_I}+\tron{Bx_I-Ay_I-By_I}\alpha +\tron{-Ax_I+Ay_I-By_I}\sigma\alpha.	
			\end{align*}
			%sửa lại hệ số ma trận 
			Let $M_\gamma= \begin{pmatrix}
				0&\frac{m}{9}\tron{2x_I-y_I}&\frac{m}{9}\tron{-x_I+2y_I}\\x_I&Ax_I-By_I-Ay_I&Bx_I-Ay_I-By_I\\y_I&Bx_I-Ay_I&-Ax_I+Ay_I-By_I
			\end{pmatrix}.$  
			
			Since all the entries in the second and third columns are multiples of $p_I$, one has that  $\det \tron{M_\gamma}$ is a multiple of $p_I^2$. Hence, $p_I^2\mid \N_{K/\Q}(\gamma)$ as $\N_{K/\Q}(\gamma) = \det(M_\gamma)$ by \cite{milne2008algebraic}.  \end{proof}

		\section{Some results related to cyclic quartic fields}\label{appendix_B}
		
		\begin{proof}[\it Proof of Lemma \ref{lem:q_i_d_even}]
			First, we prove that that $Q_{1i}, Q_{2i}$ are ideals. By Remark \ref{rem:integralbasis}.\eqref{rem:integralbasis1}, it is sufficient to show $(z_k+\sqrt{d})\beta\in Q_{kj}$. Indeed, one has 
			\begin{align*}
				(z_k+\sqrt{d})\beta &= z_k\beta + \sqrt{d}\beta \\
				&= z_k\beta +c\sigma(\beta)-b\beta\\ 
				&= \tron{z_k-b}\beta +c\sigma(\beta)\in Q_{kj}  
			\end{align*}
			for $k=1,2$. Hence $Q_{1j},Q_{2j}$ are ideals of $\mathcal{O}_F.$ These two ideals have norm $q_j$ and thus they are prime ideals. Moreover,  one has $\QQ\le K = \QQ(\sqrt{d})\le F$ and $Q_{kj}\cap \mathcal{O}_K = \mathfrak{q}_{kj}$. Hence $Q_{1j}, Q_{2j}$ are distinct. By Lemma \ref{lem:quartic_int_basis_divisor_index}, these ideals are the only prime ideals above $q_j$.
		\end{proof}

		\begin{proof}[\it Proof of Lemma \ref{lem:q_i_db_odd}]
			To prove $Q_{1j}, Q_{2j}$ are ideals, it is sufficient to prove $\frac{4t_k-1+\sqrt{d}}{2}\beta\in Q_{kj}$ and $\frac{4t_k-1+\sqrt{d}}{2}\sigma(\beta)\in Q_{kj}$ for $k=1,2$. By using Lemma \ref{lem:norm_some_ele}, we have \begin{align*}
				\frac{4t_k-1+\sqrt{d}}{2}\beta& = \tron{2t_k-1}\beta +\frac{\beta+\beta\sqrt{d}}{2} \\
				&= (2t_k-1)\beta +\frac{\beta+c\sigma(\beta)-b\beta}{2}\\
				&=\tron{2t_k-1+\frac{1-b}{2}}\beta +\frac{c}{2}\sigma(\beta)\in Q_{kj},\\
				\frac{4t_k-1+\sqrt{d}}{2}\sigma(\beta)&= (2t_k-1)\sigma(\beta)+\frac{\sigma(\beta)+\sigma(\beta)\sqrt{d}}{2} \\
				&= \tron{2t_k-1+\frac{1+b}{2}}\sigma(\beta)+\frac{c}{2}\beta\in Q_{kj},
			\end{align*} 
			for $k=1,2$. Hence $Q_{1j},Q_{2j}$ are ideals and thus they are prime as their norms are $q_j$. Moreover, $Q_{kj}\cap \mathcal{O}_K = \mathfrak{q}_{kj}$. Hence these ideals are distinct. By Lemma \ref{lem:quartic_int_basis_divisor_index}, $Q_{1j},Q_{2j}$ are two only prime ideals of $\mathcal{O}_F$ above $q_j$.
		\end{proof}
		
		The proof of Lemma \ref{lem:q_i_d_odd_b_even_ab3} is similar to Lemma \ref{lem:q_i_db_odd}.
		\begin{proof}[\it Proof of Lemma \ref{lem:qnotquadratic_ab_1_mod_4_1}]  Let $\gamma_1, \gamma_2',\gamma_3',\gamma_4$ as in Remark \ref{rem:integralbasis},iv. Let $\rho_{kj}=  \frac{4t_k-1+\sqrt{d}-\beta-\sigma(\beta)}{4}$ and 
			$\psi_{kj}= \frac{2q_j+4t_k-1+\sqrt{d}+\beta-\sigma(\beta)}{4}$. First, we prove that $Q_{kj}$ are ideals for all $k=1,2$. To do that, it is sufficient to prove that $q_j\gamma_i'$, $\frac{4t_k-1+\sqrt{d}}{2}\gamma_i'$, $ \rho_{kj}\gamma_i'$, $ \psi_{kj}\gamma_i'\in Q_{kj}$ for all $i=1,2,3,4$ and $k=1,2$.  It is obvious that $q_j\gamma_i'$; $ \frac{4t_k-1+\sqrt{d}}{2}\gamma_i'$, $\rho_{kj}$, $ \psi_{kj}\in Q_{kj}, $ for all $k=1,2$ and $i=1,2$. One has
			\begingroup
			\allowdisplaybreaks
			\begin{align*}
				q_j\gamma_3' &= \frac{q_j+1-2t_k}{2}q_j+q_j\frac{4t_k-1+\sqrt{d}}{2}-q_j\psi_{kj}\\
				q_j\gamma_4'&= t_kq_j-q_j\rho_{kj}\\
				\frac{4t_k-1+\sqrt{d}}{2}\gamma_3'&= \frac{d-\tron{4t_k-1}^2-2q_j\tron{c+1-4t_k}}{8q_j}q_j+\frac{b-c-1+8t}{4}\frac{4t_k-1+\sqrt{d}}{2}\\& \qquad-\frac{b}{2}\rho_{kj}+\frac{c-1-4t_k}{2}\psi_{kj}\\
				\frac{4t_k-1+\sqrt{d}}{2}\gamma_4'& =\frac{2bq_j-d+\tron{4t_k-1}^2}{8q_j}q_j+\frac{b+c+1}{4}\frac{4t_k-1+\sqrt{d}}{2}\\ &\qquad+\frac{-c+1-4t}{2}\rho_{kj}-\frac{b}{2}\psi_{kj}\\
				\rho_{kj}\gamma_2'&= \frac{d-\tron{4t_k-1}^2+2bq_j}{8q_j}q_j+\frac{-b-c+4t-1}{4}\frac{4t_k-1+\sqrt{d}}{2}\\&\qquad+\frac{c+1}{2}\rho_{kj}+\frac{b}{2}\psi_{kj}\\
				\psi_{kj}\gamma_2'&= \frac{d-\tron{4t_k-1}^2+2q_j\tron{c+1-4t_k}}{8q_j}q_j\\ &\qquad+\frac{-b+c-1+2q_j+4t_k}{4}\frac{4t_k-1+\sqrt{d}}{2}+\frac{b}{2}\rho_{kj}+\frac{1-c}{2}\psi_{kj}\\
				\rho_{kj}\gamma_3'& = \frac{d-\tron{4k_1-1}^2-2ab+8abt-2q\tron{b+c+1+8t}}{16q_j}q_j\\&\qquad+\frac{4t-ab-c-1}{4}\frac{4t_k+1+\sqrt{d}}{2}+\frac{c-b+1}{4}\rho_{kj}+\frac{b+c+1-4t}{4}\psi_{kj}\\
				\rho_{kj}\gamma_4'&=\frac{(4t_k-1)^2-d-2a\tron{c+d-4ct}+4bq}{16q_j}q_j+\frac{b+c-ac}{4}\frac{4t_k+1-\sqrt{d}}{2}\\&\qquad +\frac{1-c-2t_k}{2}\rho_{kj}-\frac{b}{2}\psi_{kj}\\
				\\
				\psi_{kj}\gamma_3'&=  \frac{2a\tron{c-d-4ct_k}+4q_j^2+d-\tron{4t_k-1}^2}{16q_j}q_j\\&\qquad+\frac{ac+2q_j+4t_k-1}{4}\frac{4t_k-1+\sqrt{d}}{2}+\frac{-q_j-2t_k+1}{2}\psi_{kj}\\
				\psi_{kj}\gamma_4'&= \frac{(4t_k-1)^2-d-2ab\tron{1-4t_k}+2q_j\tron{b-c-1+4t_k}}{16q_j}q_j \\& \qquad -\frac{ab-b}{4}\frac{4t_k-1+\sqrt{d}}{2}-\frac{b+c+2q_j-4t_k+1}{4}\rho_{kj}\\ &\qquad+\frac{-b+c+1}{4}\psi_{kj}.
			\end{align*}
			\endgroup
			It is not hard
			to prove all the coefficients of the above expressions are integers. Thus $Q_{1j}, Q_{2j}$ are ideals. Moreover, $Q_{kj}\cap \mathcal{O}_K = \mathfrak{q}_{kj}$ and thus $Q_{1j}\ne Q_{2j}$ and they are all prime ideals of $\mathcal{O}_F$ above $q_j$.
		\end{proof}

		To prove Lemma \ref{lem:p2_delta_odd}.\eqref{lem:p2_delta_odd2}, we again consider two cases, namely $a \equiv -c \pmod{4}$ and $a \equiv c \pmod{4}$. The proofs of both cases use the same technique, thus we only prove the first case here. The notations $\gamma_1',\gamma_2',\gamma_3',\gamma_4'$ are as defined in Remark \ref{rem:integralbasis}. One has 
		\begingroup
		\allowdisplaybreaks
		\begin{align*}
			\gamma_1'\cdot \gamma_i' &= \gamma_i', \text{ for } i=1,2,3,4\\
			\gamma_2'^2 &=\frac{d-1}{4}\gamma_1'+\gamma_2'\\
			\gamma_2'\cdot\gamma_3' &= \frac{-2b+d-1}{8}\gamma_1'+\frac{b+c+1}{4}\gamma_2'+\frac{1-c}{2}\gamma_3'+\frac{b}{2}\gamma_4'\\\gamma_2'\cdot \gamma_4' & =  \frac{-d-2c-1}{8}\gamma_1'+\frac{-b+c+1}{4}\gamma_2'+\frac{b}{2}\gamma_3'+\frac{c+1}{2}\gamma_4'\\
			\gamma_3'^2&=  \frac{-4b+2ac+2ad+d-1}{16}\gamma_1' +  \frac{b+c-ac}{4}\gamma_2'+\frac{-c+1}{2}\gamma_3'+\frac{b}{2}\gamma_4'\\ \gamma_3'\cdot\gamma_4'& = \frac{-2ab+2b-2c-d-1}{16}\gamma_1'+\frac{ab-b}{4}\gamma_2'+\frac{b+c+1}{4}\gamma_3'+\frac{-b+c+1}{4}\gamma_4'\\
			\gamma_4'^2& =\frac{-2ac+4c+2ad+d-1}{16}\gamma_1' + \frac{b+ac-c}{4}\gamma_2'-\frac{b}{2}\gamma_3'+\frac{1-c}{2}\gamma_4'.
		\end{align*}
		\endgroup
		Let $\delta =  z_1 \gamma_1'+z_2\gamma_2+z_3\gamma_3'+z_4\gamma_4'$ and $\psi =  t_1 \gamma_1'+t_2\gamma_2+t_3\gamma_3'+t_4\gamma_4'$ be arbitrary elements of $\mathcal{O}_F$. Then \begin{align*}
			\delta\cdot\psi &= S_1 \gamma_1'+S_2\gamma_2'+S_3\gamma_3'+S_4\gamma_3'
		\end{align*}where \begin{align*}
			S_1 &= z_1t_1 +z_2t_2\frac{-2b+d-1}{4}+z_2t_3\frac{-2b+d-1}{8}+z_2t_4\frac{-d-2c-1}{8}+z_3t_2\frac{-b+d-1}{8}\\&\qquad+z_3t_3\frac{-4b+2ac+2ad+d-1}{16}+z_3t_4\frac{-2ab+2b-2c-d-1}{16}+z_4t_2\frac{-d-2c-1}{8}\\&\qquad+z_4t_3\frac{-2ab+2b-2c-d-1}{16}+z_4t_4\frac{-2ac+4c+2ad+d-1}{16}\\
			S_2&= z_1t_2+z_2t_1+z_2t_2+z_2t_3\frac{b+c+1}{4}+z_2t_4\frac{-b+c+1}{4}+z_3t_2\frac{b+c+1}{4}+z_3t_3\frac{b+c-ac}{4}\\&\qquad+z_3t_4\frac{ab-b}{4}+z_4t_2\frac{-b+c+1}{4}+z_4t_3\frac{ab-b}{4}+z_4t_4\frac{b+ac-c}{4}\\
			S_3& = z_1t_3+z_2t_3\frac{1-c}{2}+z_2t_4\frac{b}{2}+z_3t_1+z_3t_2\frac{1-c}{2}\\&\qquad+z_3t_3\frac{1-c}{2}+z_3t_4\frac{b+c+1}{4}+z_4t_2\frac{b}{2}+z_4t_3\frac{b+c+1}{4}+z_4t_4\frac{-b}{2}\\
			S_4&= z_1t_4 +z_2t_3\frac{b}{2}+z_2t_4\frac{c+1}{2}+z_3t_2\frac{b}{2}\\&\qquad+z_3t_3\frac{b}{2}+z_3t_4\frac{-b+c+1}{4}+z_4t_1+z_4t_2\frac{c+1}{2}+z_4t_3\frac{-b+c+1}{4}+z_4t_4\frac{1-c}{2}
		\end{align*}
		\begin{proof}[Proof of Lemma \ref{lem:p2_delta_odd}.\eqref{lem:p2_delta_odd2}]\label{proof_lem_p2}
			To prove $2\mathcal{O}_F$ is prime, we claim that $\delta\cdot \psi\notin 2\mathcal{O}_F$ wherever $\delta\notin 2\mathcal{O}_F$ and $\psi \notin 2\mathcal{O}_F$. It is sufficient to claim that if the two tuples $\tron{t_1,t_2,t_3,t_4}$ and $\tron{z_1,z_2,z_3,z_4}$ are not simultaneously equal to $\tron{0,0,0,0} $ modulo $ 2$, then $S_1,S_2,S_3$ and $S_4$ are also not simultaneously equal to $0 \pmod 2$. Since the largest denominator of $S_1,S_2,S_3,S_4$ is $16$, one can prove this by considering the integers $a, b, c, d$ modulo $32$ and verify whether $S_1, S_2, S_3, S_4 $ are all zero modulo $ 2$ or not. It is done by using any programming language.
		\end{proof}

		\EditInfo{March 31, 2023}{August 28, 2023}{Camilla Hollanti and Lenny Fukshansky}

	\end{document}